\documentclass[11pt,reqno]{amsart}
\usepackage{amssymb}
\usepackage{verbatim}
\usepackage{amsfonts}
\usepackage{mathrsfs}
\usepackage{amsmath}
\usepackage{graphicx}
\usepackage{hyperref}
\usepackage{float}
\usepackage{epstopdf}
\usepackage{color}
\usepackage{bm}
\usepackage{comment}
\usepackage{soul}
\usepackage{cite}
\usepackage{subfigure}
\usepackage{enumerate}
\usepackage{appendix}

\allowdisplaybreaks
\newtheorem{theorem}{Theorem}[section]
\newtheorem{proposition}[theorem]{Proposition}
\newtheorem{lemma}[theorem]{Lemma}
\newtheorem{corollary}[theorem]{Corollary}

\theoremstyle{remark}
\newtheorem{remark}[theorem]{Remark}

\newtheorem{example}[theorem]{Example}

\def\dimh{\dim_{\textup{H}}}
\def\dimb{\dim_{\textup{B}}}
\def\dima{\dim_{\textup{A}}}
\def\diml{\dim_{\textup{L}}}
\def\bara{Bara\'{n}ski}

\numberwithin{equation}{section}

\begin{document}
	\title[The Hausdorff measure and uniform fibre conditions for  Bara\'nski carpets]{The Hausdorff measure and uniform fibre conditions for Bara\'nski carpets}

	\author{Hua Qiu}
	\address{School of Mathematics, Nanjing University, Nanjing, 210093, P. R. China.}
	\thanks{The research of Qiu was supported by the National Natural Science Foundation of China, grant 12471087.}
	\email{huaqiu@nju.edu.cn}
	
	\author{Qi Wang}
	\address{School of Mathematics, Nanjing University, Nanjing, 210093, P. R. China.}
	\thanks{}
	\email{602023210013@smail.nju.edu.cn}

	\subjclass[2010]{Primary 28A78,28A80}
	
	\date{}
	
	\keywords{Hausdorff measure, Hausdorff dimension, box dimension, self-affine carpets,  uniform fibres}
	
	\maketitle

	\begin{abstract}For a self-affine carpet $K$ of Bara\'{n}ski, we establish a dichotomy:  
		$$
		\text{either }\quad 0<\mathcal{H}^{\dim_{\text{H}} K}(K)<+\infty \quad \text{ or } \quad\mathcal{H}^{\dim_{\text{H}} K}(K)=+\infty.
		$$
		We introduce four types of uniform fibre condition for $K$: Hausdorff ($\textbf{u.f.H}$), Box ($\textbf{u.f.B}$), Assouad ($\textbf{u.f.A}$), and Lower ($\textbf{u.f.L}$), which are progressively stronger, with 
		$$
		\textbf{u.f.L} \Longrightarrow \textbf{u.f.A} \Longrightarrow \textbf{u.f.B} \Longrightarrow \textbf{u.f.H},
		$$
		and each implication is strict. The condition $\textbf{u.f.H}$ serves as a criterion for the dichotomy. The remaining three conditions provide an equivalent characterization for the coincidence of any two distinct dimensions. The condition $\textbf{u.f.L}$ is also equivalent to the Ahlfors regularity of $K$. As a corollary,  $\dim_{\text{H}} K=\dim_{\text{B}} K$ is sufficient but not necessary for $0<\mathcal{H}^{\dim_{\text{H}} K}(K)<+\infty$.

	\end{abstract}
	\section{Introduction}\label{sec1}

	The Hausdorff measure is a fundamental concept in the field of fractal geometry, used to measure the size of sets from a geometric perspective, particularly those with non-integer dimensions, see \cite{F90,M95}.  It is well known that for a self-similar set $K$ satisfying the open set condition (OSC), the Hausdorff measure is positive and finite in its dimension. More precisely, $K$ is Ahlfors regular, i.e. there exists a constant $c>0$ such that, denoting $\mathcal{H}^{\dimh K}$ as the Hausdorff measure in its dimension,
	$$
	c^{-1} r^{\dimh K}\leq \mathcal{H}^{\dimh K}(B(x,r)\cap K)\leq cr^{\dimh K}
	$$
	for all $x\in K$ and $0<r\leq |K|$, where $|K|$ denotes the diameter of $K$ and $B(x,r)$ is the closed ball centred at $x$ with radius $r$. However, the situation differs for self-affine sets satisfying OSC, even for those sets having regular grid structures. In 1994, Peres \cite{P94} demonstrated that for a Bedford-McMullen carpet $K$, the Hausdorff measure at its dimension is infinite, except in the rare  case where Hausdorff and box dimensions coincide. Coupled with the findings on Lalley-Gatzouras carpets \cite{LG92}, a more general class, this reveals a dichotomy for a Bedford-McMullen carpet $K$: 
	\begin{equation}\label{e92}
		\text{ either }\quad 0<\mathcal{H}^{\dimh K}(K)<+\infty \quad \text{ or }\quad \mathcal{H}^{\dimh K}(K)=+\infty.
	\end{equation}
	The former occurs if and only if $K$ possesses uniform fibres. One of the main aims in this paper is to extend the consideration to  \bara's carpets \cite{B07}: Does such a dichotomy hold, and if so, how can it be characterized?
	\vspace{0.2cm}
	
	Take a unit square and divide it up into rectangles via a finite number of vertical or horizontal lines;  formulate an \textit{iterated function system} (IFS) which maps the unit square onto a collection of chosen  rectangles through  orientation preserving linear contractions. The attractors of such IFSs are called \bara's  carpets \cite{B07}. Precisely, let $r,s,d\geq 2$ with $d\leq rs$ be three integers. For $i=1,\cdots,r$ and $j=1,\cdots,s$, choose $0<a_i, b_j<1$ such that $\sum_{i=1}^r a_i= 1$ and $\sum_{j=1}^s b_j= 1$.  Divide the  unit square $[0,1]^2$ into $r$ vertical strips of widths $a_1,\cdots, a_r$, and $s$ horizontal strips of heights $b_1, \cdots, b_s$. Let  
	$$
	\mathcal{J}=\big\{(i_1,j_1),\cdots, (i_d,j_d)\big\}\subseteq \{ 1, \cdots, r\} \times \{1,\cdots,s\}
	$$
	be a subset. 
	For $l=1,\cdots,d$, define the affine mapping $\psi_l$ as follows: 
	$$
	\psi_l \left( \begin{array}{c}
		x\\y
	\end{array} \right) 
	= \left(
	\begin{array}{cc}
		a_{i_l} & 0 \\
		0 & b_{j_l}
	\end{array} \right)
	\left( \begin{array}{c}
		x\\y
	\end{array} \right) 
	+
	\left( \begin{array}{c}
		\sum_{i=1}^{i_l-1} a_i\\ \sum_{j=1}^{j_l-1}b_j
	\end{array} \right) .
	$$
	Then, the set $\{\psi_1,\cdots, \psi_d\}$ forms a self-affine IFS, see Figure \ref{f2} (left) for an illustration. Denote $K$ as its attractor, which is the unique non-empty compact set satisfying $K=\bigcup_{l=1}^d \psi_l(K)$. We refer to $\{\psi_1,\cdots, \psi_d\}$ as a \textit{\bara~IFS} and $K$ as a \textit{Bara\'{n}ski carpet}. 
	\begin{figure}[htbp]
		\centering
		\subfigure{
			\includegraphics[width=0.3\textwidth]{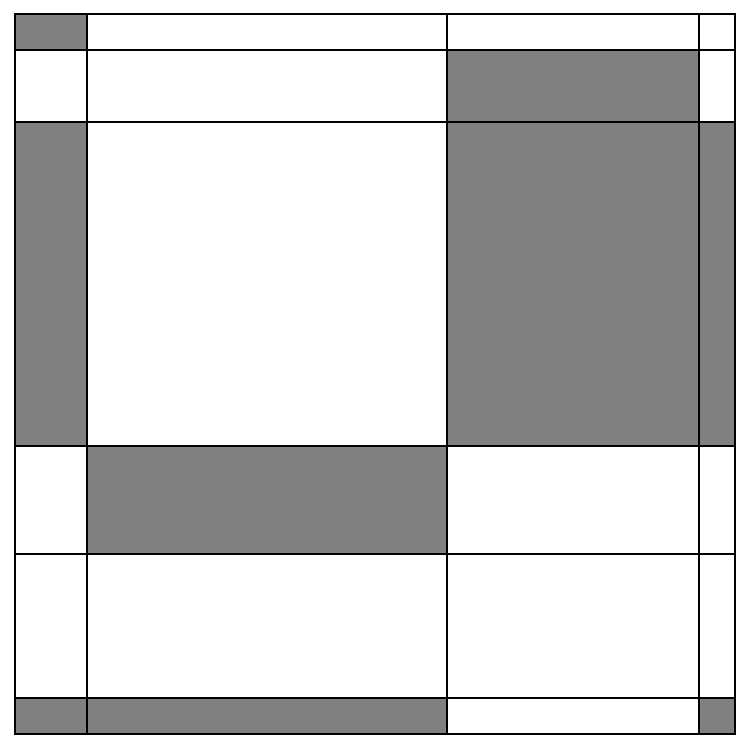}
		}
		\subfigure{
			\includegraphics[width=0.3\textwidth]{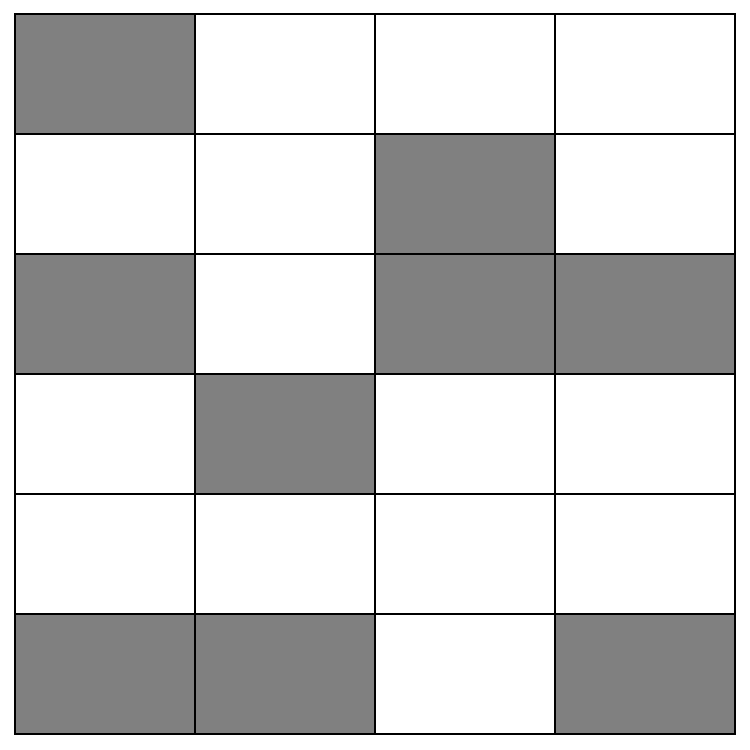}
		}
		\subfigure{
			\includegraphics[width=0.3\textwidth]{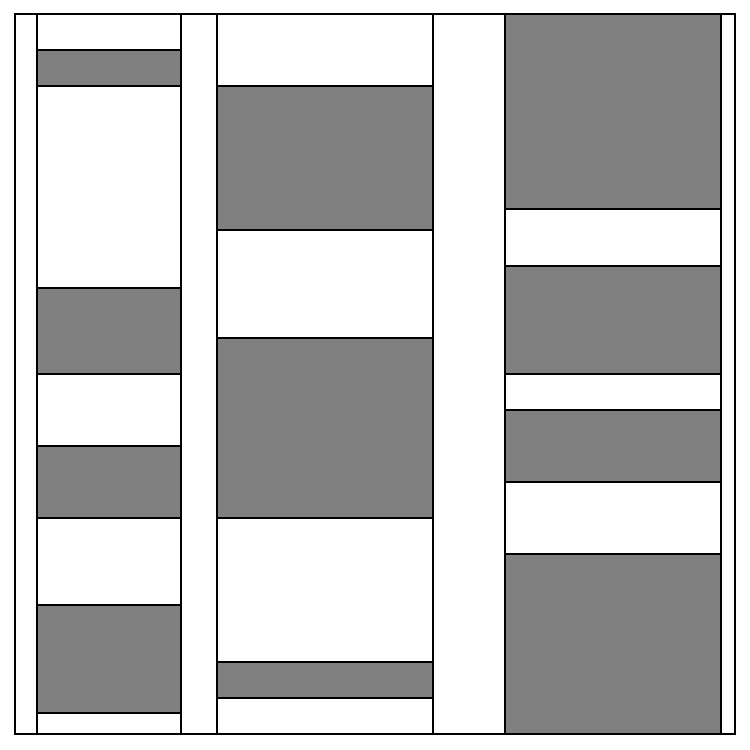}
		}
		\caption{Generating templates for a \bara~carpet (\textit{left}), a Bedford-McMullen carpet (\textit{middle}) and a Lalley-Gatzouras carpet (\textit{right}).}
		\label{f2}
	\end{figure}

	Note that if a \bara~carpet $K$ degenerates to a self-similar set (for example, $a_{i_l}=b_{j_l}$ for all $l$), since it always satisfies the OSC, it is Ahlfors regular with
	$$
	\dimh K=\dimb K\quad \text{ and }\quad 0<\mathcal{H}^{\dimh K}(K)<+\infty,
	$$ where we use $\dimh K$ and $\dimb K$ to represent the Hausdorff and box dimensions of $K$, respectively (see \cite{F90,M95} for general discussions).
	So to avoid trivial cases, in the following we always assume that 
	\begin{equation}\label{e38}
		\text{\textbf{Assumption:}}  \left\{
		\begin{aligned}
			&\text{ there exists }l\in \{1,\cdots,d\} \text{ such that } a_{i_l}\neq b_{j_l}, 
			\\
			&\text{ and } K \text{ is not contained within any vertical or horizontal line}.
		\end{aligned}\right. 
	\end{equation}

	The \bara\  carpets constitute an important and extensively studied planar self-affine sets, characterized by diagonal generating linear contractions and a certain geometric alignment of cylinders. Historically, the initial study among these sets is on the so-called \textit{Bedford-McMullen carpets}, where $r\neq s$ (with the convention $r<s$ without loss of generality), $a_i=r^{-1}$ and $b_j=s^{-1}$ for all $i$ and $j$, as shown in Figure \ref{f2} (middle). Bedford \cite{B84} and McMullen \cite{M84} independently calculated the Hausdorff and box dimensions, and indicated that $\dimh K=\dimb K$ whenever the number of chosen rectangles in each ``nonempty'' column is uniform, i.e. $K$ has uniform fibres. Later,  Lalley and Gatzouras \cite{LG92} extended the study to a more flexible setting by maintaining the column structure, allowing distinct values for $a_i$, and permitting rectangles along each column to shift vertically, yet still requiring that the height of each rectangle remains less than its width, as depicted in Figure \ref{f2} (right). For a \textit{Lalley-Gatzouras carpet} $K$, it was demonstrated  that the Hausdorff  and box dimensions, $\dimh K$  and $\dimb K$, can be derived as the maximum values from certain variational formulas over Bernulli measures supported on $K$ (the formula for $\dimb K$ was originally expressed in a closed form of the ratios of the IFS \cite{LG92}, while an equivalent variational expression was presented later by Feng and Wang \cite[Theorem 1]{FW05} on a more general box-like self-affine carpets). Notably, 
	\begin{equation}\label{e91}
		K \text{ has uniform fibres} \iff \dimh K=\dimb K  \iff 0<\mathcal{H}^{\dimh K}(K)<+\infty.
	\end{equation}
	
	The key distinction between the \bara\  and Lalley-Gatzouras classes is that in the \bara\  class, the larger ratio of each contraction does not necessarily align in the same  direction, whether horizontal or vertical, which significantly increases the complexity. In this regard, the \bara\  class can be considered   an essential extension of Lalley-Gatzouras's. \bara\  derived the formulas for the Hausdorff and box dimensions of such fractals by extending the analysis to account for both directions, although it is considerably more intricate. 
	
	\vspace{0.2cm}
	
	At first glance, one may think that the implications in \eqref{e91} hold true for \bara\ carpets, at least for ``$\dimh K=\dimb K  \iff 0<\mathcal{H}^{\dimh K}(K)<+\infty$''. However, for a \bara\ carpet $K$, the challenge lies in  formulating the condition for ``uniform fibres'', given that there are two directions to consider: \textit{uniform vertical fibres} and \textit{uniform horizontal fibres} (for precise definitions, refer to Section \ref{sec2}). A natural question arises: which direction predominantly influences the characterization? 
	
	To address this, we will introduce two types of uniform fibre condition: the \textit{uniform fibre condition of Hausdorff type} \eqref{H}, and the \textit{uniform fibre condition of Box type} \eqref{B}, which are closely related to the expressions of Hausdorff and box dimensions, respectively. Roughly speaking, we say that \ref{H}  holds if $K$ has uniform vertical (resp. horizontal) fibres whenever the Hausdorff dimension of $K$ is achieved by its variational formula through projecting Bernulli measures vertically (resp. horizontally). The case is similar for \ref{B}.  For their exact definitions, see Section \ref{sec2}. We point out that the former is equivalent to ``$0<\mathcal{H}^{\dimh K}(K)<+\infty$'', and the later is equivalent to ``$\dimh K=\dimb K $''. 
	
	\begin{theorem}\label{th1} 
		Let $K$ be a \bara~carpet. Then
		\begin{align}
			\ref{H} &\iff0<\mathcal{H}^{\dimh K}(K)<+\infty, \tag{a} \label{e40}
			\\
			\neg \ref{H} &\iff \mathcal{H}^{\dimh K}(K)=+\infty, \tag{b} \label{e41}
			\\
			\ref{B} &\iff \dimh K =\dimb K.\tag{c}\label{e42}
		\end{align}
	\end{theorem}
	
	For Theorem \ref{th1}, it suffices to prove the direction ``$\Longrightarrow$'' in \eqref{e40} and \eqref{e41} as the other direction follows by contraposition.
	We will proceed to prove Theorem \ref{th1} in the order of the direction ``$\Longrightarrow$'' in \eqref{e41}, the two directions in \eqref{e42}, and the direction ``$\Longrightarrow$'' in \eqref{e40}.
	
	The theorem also demonstrates that the dichotomy \eqref{e92} for Hausdorff measure continues to hold for \bara\  carpets, which is determined by condition \ref{H}. However, in general, 
	$$\dimh K=\dimb K  \Longleftarrow\mkern-21.5mu\backslash \quad0<\mathcal{H}^{\dimh K}(K)<+\infty,$$ as stated in the following corollary.
	
	\begin{corollary}\label{cor1}
		Let $K$ be a \bara~carpet. Then, $\dimh K=\dimb K$ is a sufficient but not necessary condition for  $0<\mathcal{H}^{\dimh K}(K)<+\infty$.

	\end{corollary}

	In other words, it is possible that
	\begin{equation}\label{e37}
		\dimh K <\dimb K\quad\text{ and }\quad 0<\mathcal{H}^{\dimh K}(K)<+\infty.
	\end{equation}
	
	From Theorem \ref{th1} and Corollary \ref{cor1}, it is evident that condition \ref{B} is generally stronger than condition \ref{H}. However, they are equivalent when $K$ degenerates to a Lalley-Gatzouras carpet, as noted in Remark \ref{re2}. Consequently, Theorem \ref{th1} reduces to \eqref{e91} for Lalley-Gatzouras carpets, a result that is already known.
	
	\vspace{0.2cm}
	
	The second aim in this paper is to paint a complete picture of the comparisons between any two distinct dimensions of \bara\  carpets, including the Hausdorff and box dimensions, as well as the  Assouad and lower dimensions, denoted as $\dima K$ and $\diml K$, respectively. These latter two dimensions form a dual ``dimension pair'', emphasizing the local geometric properties of fractals and reflecting the  scaling laws of the most and least dense parts of  the sets (for a general discussion, see \cite{F21}).
	
	Recall that in 2011, Mackay \cite{M11} initiated the study of Assouad dimension of  Lalley-Gatzouras carpets, employing a technique of constructing weak tangents. The approach was later extended by Fraser \cite{F14} to Bara\'{n}ski carpets. Dually, the lower dimension (introduced by Larman  \cite{Lar67}) of \bara\ sets was also determined in  \cite{F14}. Analogous to  condition \ref{H} and \ref{B}, we will introduce two more uniform fibre conditions: \textit{the uniform fibre condition of Assouad type} \eqref{A}, and \textit{the uniform fibre condition of Lower type} \eqref{L}, which correspond to the expressions of the Assouad and lower dimensions of \bara\ carpets. See Section \ref{sec2} for their exact definitions.
	
	\begin{theorem}\label{th3}
		Let $K$ be a \bara~carpet. Then
		\begin{align}
			\ref{A}  &\iff \dimb K=\dima K  \iff  \dimh K =\dima K,\tag{a}\label{e11} 
			\\
			\ref{L}  &\iff  \diml K=\dimh K \iff  \diml K=\dima K \iff K \text{ is Ahlfors regular.} \tag{b}\label{e12}
		\end{align}
		
		
		
	\end{theorem}

	Utilizing Theorems \ref{th1} and \ref{th3}, we can finally establish the equivalent condition for the coincidence of each pair of dimensions of \bara~carpets. We remark that conditions \ref{H}, \ref{B}, \ref{A} and \ref{L} are increasingly stronger (Proposition \ref{prop5}), and all four are equivalent when $K$ degenerates to a Lalley-Gatzouras carpet (Remark \ref{re2}). It is also worth mentioning that Theorems \ref{th1} and \ref{th3} hold for the full Lalley-Gatzouras class in a  analogous way, although Lalley-Gatzouras class is not generally contained within the \bara\ class.
	
	Another natural question arises: what relationships are possible between the dimensions? See Fraser \cite[Question 4.5]{F14}. With a little more work, the following corollary provides an answer.
	
	\begin{corollary}\label{cor4}
		Let $K$ be a \bara~carpet. We have:
		\begin{enumerate}[(a)]
			\item \label{e20}
			$\diml K=\dimh K=\dimb K=\dima K$ is possible,
			\item \label{e21}
			$\diml K=\dimh K=\dimb K<\dima K$ is not possible,
			\item \label{e22}
			$\diml K=\dimh K<\dimb K=\dima K$ is not possible,
			\item \label{e23}
			$\diml K=\dimh K<\dimb K<\dima K$ is not possible,
			\item \label{e24}
			$\diml K<\dimh K=\dimb K=\dima K$ is possible,
			\item \label{e25}
			$\diml K<\dimh K=\dimb K<\dima K$ is possible,
			\item \label{e26}
			$\diml K<\dimh K<\dimb K=\dima K$ is not possible,
			\item \label{e27}
			$\diml K<\dimh K<\dimb K<\dima K$ is possible.
		\end{enumerate}
	\end{corollary}
	\begin{proof}
		The cases \eqref{e20},\eqref{e21},\eqref{e24} and \eqref{e27} have already been addressed by Fraser, as referenced in  \cite[Question 4.5]{F14}. The cases \eqref{e22} and \eqref{e23} follow directly from Theorem \ref{th3}-\eqref{e12}. The case \eqref{e25} is derived from Example \ref{ex2}, which will be presented later in Section \ref{sec8}. Lastly, the case \eqref{e26} is a consequence of Theorem \ref{th3}-\eqref{e11}.
	\end{proof}
	
	\begin{proposition}\label{prop5}
		Conditions \ref{H}, \ref{B}, \ref{A}, and \ref{L} satisfy 
		$$
		\ref{L} \Longrightarrow \ref{A} \Longrightarrow \ref{B} \Longrightarrow \ref{H},
		$$
		and these implications are strict.
	\end{proposition}
	\begin{proof}
		The strict implication $\ref{L} \Longrightarrow \ref{A}$ follows from Theorem \ref{th3} and  Corollary \ref{cor4}-\eqref{e24}. The strict implication $\ref{A} \Longrightarrow \ref{B}$ follows from Theorems \ref{th1} and  \ref{th3}, along with Corollary \ref{cor4}-\eqref{e25}.
		The strict implication $\ref{B} \Longrightarrow \ref{H}$ follows from Theorem \ref{th1} and Corollary \ref{cor1}.
	\end{proof}
	
	\begin{remark}
		For a \bara\ carpet, all the four conditions can be verified  through calculations on the parameters in its IFS. This is apparent for \ref{B}, \ref{A} and \ref{L}, while for \ref{H}, it is also true via its equivalent characterization \ref{H'} (Proposition \ref{le15}) introduced later in Section \ref{sec3}.
	\end{remark}
	
	Before ending this section, let us look at the higher dimensional case. In 2017, Das and Simmons \cite{DS17} showed that the variational formula for the Hausdorff dimension of \bara\ sponges (which are higher dimensional \bara\ carpets) does not necessarily hold true via Bernoulli measures, but rather a broader class known as  \textit{pseudo-Bernoulli measures}. In 2023, Kolossváry \cite{K23} determined the box dimension as a complex Ledrappier-Young type variational formula. It remains unclear how to extend the considerations in this paper to \bara\ sponges.
	
	\vspace{0.2cm}

	The paper is organized as follows. In Section \ref{sec2}, we review the formulas for the Hasudorff, box, Assouad and lower dimensions of the \bara~carpets, and propose four conditions: \ref{H}, \ref{B}, \ref{A}, and \ref{L}. In Section \ref{sec3}, we prepare some lemmas and provide an equivalent condition \ref{H'} to \ref{H} in Subsection \ref{sec3.1}, and review the concept of ``approximate squares'' in Subsection \ref{sec3.2}. We first prove ``$\Longrightarrow$'' in Theorem \ref{th1}-\eqref{e41}  in Section \ref{sec4}. The proof of Theorem \ref{th1}-\eqref{e42}  is in Section \ref{sec6}, and  ``$\Longrightarrow$'' in Theorem \ref{th1}-\eqref{e40} is in Section \ref{sec5}. The proof of Corollary \ref{cor1} is postponed to  Section \ref{sec8}, where we also provide an example, Example \ref{ex1}, on which \eqref{e37} holds. The proof of Theorem \ref{th3} is in Section \ref{sec7}.

	\section{Conditions}\label{sec2}
	
	In this section, we review the Hausdorff, box, Assouad and lower dimensions of \bara~carpets, and introduce the four types of uniform fibre condition \ref{H}, \ref{B}, \ref{A} and \ref{L}. 
	
	Before stating the formulas for the Hausdorff and box dimensions of a \bara~carpet $K$, let us look at the following notations. 
	
	For a set $A$, denote $\text{int}(A)$ as the \textit{interior} of $A$ and $\#A$  as the \textit{cardinality} of $A$.  Let $\pi_1,\pi_2:\mathbb{R}^2\to \mathbb{R}$ be the \textit{orthogonal projections} from $\mathbb{R}^2$ to the $x$-axis and $y$-axis, respectively.
	For a \bara\ carpet, let 
	$$
	\mathcal{S}=\left \{\bm{q}=(q_1,\cdots, q_d)\in \mathbb{R}^d:q_l\geq 0,  \sum_{l=1}^d q_l=1\right \},
	$$
	and for $i=1,\cdots, r$, $j=1,\cdots, s$ and $\bm{q}=(q_1,\cdots, q_d)\in \mathcal{S}$, define
	$$
	\begin{aligned}
		I_i&=\big\{l\in \{ 1,\cdots, d\}:i_l=i\big\}, \quad  & & R_i(\bm{q})=\sum_{l\in I_i}q_l,\\
		J_j&=\big\{l\in \{ 1,\cdots, d\}:j_l=j\big\}, & & S_j(\bm{q})=\sum_{l\in J_j} q_l,
	\end{aligned}
	$$
	and 
	$$
	\begin{aligned}
		\mathcal{S}_1&=\left\{ \bm{q}\in \mathcal{S}:\sum_{i=1}^r R_i(\bm{q})\log a_i \geq \sum_{j=1}^s S_j(\bm{q})\log b_j\right\},\\
		\mathcal{S}_2&=\left\{ \bm{q}\in \mathcal{S}:\sum_{i=1}^r R_i(\bm{q})\log a_i \leq  \sum_{j=1}^s S_j(\bm{q})\log b_j\right\}.
	\end{aligned}
	$$
	Recall that 
	$$
	\mathcal{J}=\big \{(i_1,j_1),\cdots,(i_d,j_d)\big\} \subseteq \{1,\cdots,r\}\times \{1,\cdots, s \big \}.
	$$
	Let 
	$$
	\mathcal{J}_1=\big\{i\in \{1,\cdots, r\}:I_i\neq \emptyset\big\},\quad \text{ and }\quad 
	\mathcal{J}_2=\big\{j\in \{1,\cdots, s\}:J_j\neq \emptyset\big\}.
	$$ 
	Then $\mathcal{J}\subseteq \mathcal{J}_1\times \mathcal{J}_2$ and  $\mathcal{J}_1=\pi_1(\mathcal{J})$, $\mathcal{J}_2=\pi_2(\mathcal{J})$.

	For $\bm{q}\in \mathcal{S}$, define a function $g:\mathcal{S}\to \mathbb{R}$ by 
	$$
	g(\bm{q})=\left \{\begin{aligned}
		&\frac{\sum_{i=1}^r R_i(\bm{q})\log R_i(\bm{q})}{\sum_{i=1}^r R_i(\bm{q}) \log a_i}+ \frac{\sum_{l=1}^d q_l \log q_l -\sum_{i=1}^r R_i(\bm{q})\log R_i(\bm{q})}{\sum_{j=1}^s S_j(\bm{q})\log b_j}, \quad & & \text{ if } \bm{q}\in \mathcal{S}_1, \\
		&\frac{\sum_{j=1}^s S_j(\bm{q}) \log S_j(\bm{q})}{\sum_{j=1}^s S_j(\bm{q}) \log b_j}+\frac{\sum_{l=1}^d q_l \log q_l -\sum_{j=1}^s S_j(\bm{q})\log S_j(\bm{q})}{\sum_{i=1}^r R_i(\bm{q})\log a_i}, & &\text{ if } \bm{q} \in \mathcal{S}_2 \setminus \mathcal{S}_1,
	\end{aligned} \right.
	$$
	where $0\cdot \log 0=0 $ for convention. 
	
	Let $t_1,t_2$ be the unique real numbers satisfying 
	\begin{equation}\label{e1}
		\sum_{i\in \mathcal{J}_1} a_i^{t_1}=1, \quad\quad \sum_{j\in \mathcal{J}_2} b_j^{t_2}=1,
	\end{equation}
	and let $D_1,D_2$ be the unique real numbers satisfying 
	\begin{equation}\label{e10}
		\sum_{l=1}^d a_{i_l}^{t_1} b_{j_l}^{D_1-t_1}=1, \quad\quad \sum_{l=1}^d b_{j_l}^{t_2} a_{i_l}^{D_2-t_2}=1.
	\end{equation}

	\begin{proposition}(\cite[Theorems A,B]{B07})\label{prop2}
		Let $K$ be a \bara~carpet. We have
		$$
		\dimh K=\max_{\bm{q}\in \mathcal{S}}g(\bm{q}), \quad \dimb K=\max \{ D_1,D_2\}.
		$$
	\end{proposition}
	For a \bara~carpet $K$, we say $K$ has \textit{uniform vertical fibres} if there is  $t\geq0$ such that 
	\begin{equation}\label{e39}
		\sum_{l\in I_i} b_{j_l}^t=1,\quad      \text{ for all } i\in \mathcal{J}_1.
	\end{equation}
	Similarly, we say $K$ has \textit{uniform horizontal fibres} if there is  $t\geq 0$ such that 
	$$
	\sum_{l\in J_j} a_{i_l}^t=1,\quad \text{ for all } j\in \mathcal{J}_2.
	$$
	Note that if $K$ has uniform vertical fibres, then for every $x$ in $\pi_1(K)$, we have 
	$$
	\dimb \big( K\cap (\{x\}\times[0,1]) \big)=t,
	$$
	where $t$ is the number satisfying \eqref{e39}.
	\vspace{0.2cm}
	
	Now we give the definitions of \textit{uniform fibre condition of Hausdorff type} (\ref{H} for short)   and \textit{uniform fibre condition of Box type} (\ref{B} for short): 
	\begin{equation}
		K \text{ has } \left\{
		\begin{aligned}
			&\text{uniform vertical fibres}, & & \text{ if }\bm{q}\in \mathcal{S}_1 \text{ for some }\bm{q} \text{ with }g(\bm{q})=\dimh K, \\
			&\text{uniform horizontal fibres}, & & \text{ if }\bm{q}\in\mathcal{S}_2 \text{ for some }\bm{q} \text{ with }g(\bm{q})=\dimh K,
		\end{aligned} \right. \tag{\textbf{u.f.H}}\label{H}
	\end{equation}
	and 
	\begin{equation}
		K \text{ has } \left\{
		\begin{aligned}
			&\text{uniform vertical fibres}, & & \text{ if } D_1>D_2, \\
			&\text{uniform horizontal fibres}, & & \text{ if }D_1<D_2,\\
			&\text{uniform vertical \textbf{or} horizontal fibres}, & & \text{ if }D_1=D_2.
		\end{aligned} \right. \tag{\textbf{u.f.B}}\label{B}
	\end{equation}
	Note that condition \ref{H} is not easy to verify, but it is  intuitive. We will propose an equivalent condition, \ref{H'}, in Section \ref{sec3.1}, which resembles condition \ref{B}.
	\vspace{0.2cm}	
	
	Next, in order to propose conditions \ref{A} and \ref{L}, we introduce the formulas of Assouad and lower dimensions of $K$ as presented in \cite{F14}. For $l=1,\cdots, d$, let $S_{1,l},S_{2,l}$ be the unique solutions of 
	\begin{equation}\label{e15}
		\sum_{m\in I_{i_l}} b_{j_{m}}^{S_{1,l}}=1,\quad \quad 
		\sum_{m\in J_{j_l}} a_{i_{m}}^{S_{2,l}}=1.
	\end{equation}
	It follows from the definition that for $l\neq m$, if ${i_l}={i_m}$, then $S_{1,l}=S_{1,m}$. Furthermore, $K$ has uniform vertical (resp. horizontal) fibres if and only if  all $S_{1,l}$ (resp. $S_{2,l})$ are the same.

	Using the terminology by Fraser \cite{F14},  we say a \bara~carpet $K$ is of \textit{horizontal type} if $a_{i_l}\geq b_{j_l}$ for all $l=1,\cdots, d$; of \textit{vertical type} if $a_{i_l}\leq b_{j_l}$ for all $l=1,\cdots, d$; and of \textit{mixed type} otherwise. Recall $t_1=\dimh \pi_1(K)$ and $t_2= \dimh \pi_2(K)$, which satisfy equation \eqref{e1}. Set 
	\begin{equation}\label{e79} \left\{
		\begin{aligned}
			E_1&=t_1+\max_{l\in \{1,\cdots, d\}} S_{1,l},& & E_2=-1, \quad & & \text{ if } K \text{ is of horizontal type},
			\\
			E_1&=-1,\quad& & E_2=t_2+\max_{l\in \{1,\cdots, d\}} S_{2,l}, & & \text{ if } K \text{ is of vertical type},
			\\
			E_1&=t_1+\max_{l\in \{1,\cdots, d\}} S_{1,l},\quad& & E_2=t_2+\max_{l\in \{1,\cdots, d\}} S_{2,l}, & & \text{ if } K \text{ is of mixed type},
		\end{aligned} \right.
	\end{equation}
	and 
	\begin{equation}\label{e108} \left\{
		\begin{aligned}
			F_1&=t_1+\min_{l\in \{1,\cdots, d\}} S_{1,l},\quad& & F_2=3, \quad & & \text{ if } K \text{ is of horizontal type},
			\\
			F_1&=3,\quad& & F_2=t_2+\min_{l\in \{1,\cdots, d\}} S_{2,l}, \quad & & \text{ if } K \text{ is of vertical type},
			\\
			F_1&=t_1+\min_{l\in \{1,\cdots, d\}} S_{1,l},\quad& & F_2=t_2+\min_{l\in \{1,\cdots, d\}} S_{2,l}, \quad & & \text{ if } K \text{ is of mixed type}.
		\end{aligned} \right.
	\end{equation}
	\begin{proposition}(\cite[Theorems 2.12, 2.13]{F14})\label{prop3}
		Let $K$ be a \bara~carpet. We have 
		$$
		\dima K=\max \{E_1,E_2\}, \quad
		\diml K=\min \{F_1,F_2\}.
		$$
	\end{proposition}
	
	For a \bara~carpet $K$, we say it satisfies the \textit{uniform fibre condition of Assouad type} (\ref{A} for short) if 
	\begin{equation}
		K \text{ has} \left\{ 
		\begin{aligned}
			&\text{uniform vertical fibres}, & & \text{ if } E_1>E_2, \\
			&\text{uniform horizontal fibres}, & & \text{ if }E_1<E_2,\\
			&\text{uniform vertical \textbf{or} horizontal fibres}, & & \text{ if }E_1=E_2,
		\end{aligned}
		\right.\tag{\textbf{u.f.A}}\label{A}
	\end{equation}
	or \textit{uniform fibre condition of Lower type} (\ref{L} for short) if 
	\begin{equation}
		K \text{ has} \left\{ 
		\begin{aligned}
			&\text{uniform vertical fibres}, & & \text{ if }K\text{ is of horizontal type}, \\
			&\text{uniform horizontal fibres}, & & \text{ if }K\text{ is of vertical type},\\
			&\text{uniform vertical \textbf{and} horizontal fibres \textbf{and} } F_1=F_2, & & \text{ if } K \text{ is of mixed type}.
		\end{aligned}
		\right.\tag{\textbf{u.f.L}}\label{L}
	\end{equation}

	\begin{remark}\label{re2}
		Note that when $K$ is of horizontal or vertical type, there is a dichotomy by Fraser \cite[Corollary 2.14]{F14}:
		\begin{equation}\label{e102}
			\begin{aligned}
				\text{ either }\quad& \diml K< \dimh K <\dimb K <\dima K,
				\\
				\text{ or }\quad& \diml K =\dimh K =\dimb K =\dima K.
			\end{aligned}
		\end{equation}
		Indeed, in this case we can see that 
		either $\mathcal{S}_1=\mathcal{S}$ or $\mathcal{S}_2=\mathcal{S}$, which implies that $\ref{H} \Longrightarrow \ref{L}$. Combining this with Proposition \ref{prop5}, we further see that  
		$$
		\ref{L} \iff \ref{A} \iff \ref{B} \iff \ref{H}.
		$$
		So by Theorems \ref{th1} and \ref{th3}, we can strengthen equation \eqref{e102} to
		$$
		\begin{aligned}
			\text{ either }\quad& \diml K< \dimh K <\dimb K <\dima K \text{ and } \mathcal{H}^{\dimh K }(K)=+\infty,
			\\
			\text{ or }\quad & \diml K =\dimh K =\dimb K =\dima K \text{ and } 0< \mathcal{H}^{\dimh K}(K)<+\infty.
		\end{aligned}
		$$
	\end{remark}

	\section{Observations on the function $g$ and approximate squares}\label{sec3}
	For a \bara~carpet $K$, recall that in Proposition \ref{prop2}, $\dimh K=\max_{\bm{q}\in \mathcal{S}} g(\bm{q})$. The expression of function $g$ is complex. In Subsection \ref{sec3.1}, we prepare some lemmas to analyze the property of the function $g$. In Subsection \ref{sec3.2}, We review the ``approximate square'' method to help us estimate the Hausdorff measure of $K$.
	
	\subsection{Some observations on $g$}\label{sec3.1}
	For $\bm{q}\in \mathcal{S}$, define 
	$$ \left \{
	\begin{aligned}
		RR(\bm{q}) :&=\sum_{i=1}^r R_i(\bm{q}) \log R_i(\bm{q}),\quad
		\\
		SS(\bm{q}) :&=\sum_{j=1}^s S_j(\bm{q}) \log S_j(\bm{q}),\quad
		\\
		QQ(\bm{q}) :&=\sum_{l=1}^d q_l \log q_l,
	\end{aligned} \right.
	$$
	and 
	\begin{equation} \label{e83}
		\left\{
		\begin{aligned}
			RA(\bm{q}) :&=\sum_{i=1}^r R_i(\bm{q}) \log a_i=\sum_{l=1}^d q_l \log a_{i_l},\quad
			\\
			SB(\bm{q}) :&=\sum_{j=1}^s S_j(\bm{q}) \log b_j=\sum_{l=1}^d q_l \log b_{j_l}.
		\end{aligned} \right.
	\end{equation}
	Define $g_1,g_2:\mathcal{S}\to \mathbb{R}$ by  
	\begin{equation}\label{e6}
		\begin{aligned}
			g_1(\bm{q})&:=\frac{RR(\bm{q})}{RA(\bm{q})}+\frac{QQ(\bm{q})-RR(\bm{q})}{SB(\bm{q})},
			\\
			g_2(\bm{q})&:=\frac{SS(\bm{q})}{SB(\bm{q})}+\frac{QQ(\bm{q})-SS(\bm{q})}{RA(\bm{q})},
		\end{aligned}
	\end{equation}
	so that the function $g$ can be written as 
	\begin{equation}\label{e84}
		g(\bm{q})=\left \{\begin{aligned}
			&g_1(\bm{q}), \quad & & \text{ if } \bm{q}\in \mathcal{S}_1, \\
			&g_2(\bm{q}), & &\text{ if } \bm{q} \in \mathcal{S}_2 \setminus \mathcal{S}_1.
		\end{aligned} \right.
	\end{equation}
	Let 
	\begin{equation}\label{e80}
		G_1:=\max_{\bm{q}\in \mathcal{S}} g_1(\bm{q})\quad \text{ and } \quad G_2:=\max_{\bm{q}\in \mathcal{S}} g_2(\bm{q}).
	\end{equation}
	We introduce another \textit{uniform fibre condition of Hausdorff type} (\ref{H'} for short):
	\begin{equation}
		K \text{ has } \left\{
		\begin{aligned}
			&\text{uniform vertical fibres,} & & \text{ if } G_1>G_2, \\
			&\text{uniform horizontal fibres,} & & \text{ if }G_1<G_2,\\
			&\text{uniform vertical \textbf{and} horizontal fibres,} & & \text{ if }G_1=G_2,
		\end{aligned} \right. \tag{\textbf{u.f.H$'$}}\label{H'}
	\end{equation}
	which is similar in form to \ref{B}.
	\begin{proposition}\label{le15}
		The conditions \ref{H} and \ref{H'} are equivalent.
	\end{proposition}
	Before proving Proposition \ref{le15}, we present some lemmas.

	\begin{lemma}\label{le1}
		For $\bm{q}=(q_1,\cdots,q_d) \in \emph{int}(\mathcal{S})$, we have
		\begin{equation}\label{e5}
			QQ(\bm{q})\geq RR(\bm{q})+SS(\bm{q})
		\end{equation}
		and equality holds if and only if $\mathcal{J}_1 \times \mathcal{J}_2 =\mathcal{J}$ and $q_l=R_{i_l}(\bm{q})\cdot S_{j_l}(\bm{q})$, for all $l=1,\cdots,d$. 
	\end{lemma}
	\begin{proof}
		The intuition of this lemma is the entropy inequality for a join of two partitions.
		For $(i,j)\in \mathcal{J}_1 \times \mathcal{J}_2$, define
		$$
		q_{(i,j)} =\left \{ \begin{aligned}
			&q_l,\quad & & (i,j)=(i_l,j_l) \text{ for some }l=1,\cdots, d,\\
			&0, & & \text{ otherwise.}
		\end{aligned} \right . 
		$$
		Write $\bm{q}'=\big( q_{(i,j)} \big)_{(i,j)\in \mathcal{J}_1 \times \mathcal{J}_2}$, $R_i(\bm{q}')=\sum_{j\in \mathcal{J}_2} q_{(i,j)}$ and  $S_j(\bm{q}')=\sum_{i\in \mathcal{J}_1} q_{(i,j)}$ for $i\in \mathcal{J}_1,j\in \mathcal{J}_2$.
		Then $R_i(\bm{q}')=R_i(\bm{q})>0$ and $S_j(\bm{q}')=S_j(\bm{q})>0$ for all $i\in\mathcal{J}_1 $, $j\in \mathcal{J}_2$. So, 
		\begin{align*}
			QQ(\bm{q}) &=\sum_{(i,j)\in \mathcal{J}_1\times \mathcal{J}_2 } q_{(i,j)}\log q_{(i,j)} =\sum_{i\in \mathcal{J}_1} \sum_{j\in \mathcal{J}_2} \frac{q_{(i,j)}}{R_i(\bm{q}')}\cdot R_i(\bm{q}') \log \left(\frac{q_{(i,j)}}{R_i(\bm{q}')}\cdot R_i(\bm{q}')\right)
			\\
			&=\sum_{j\in \mathcal{J}_2}  \sum_{i\in \mathcal{J}_1} R_i(\bm{q}') \cdot \frac{q_{(i,j)}}{R_i(\bm{q}')} \log \left(\frac{q_{(i,j)}}{R_i(\bm{q}')}\right)+\sum_{i\in \mathcal{J}_1}  R_i(\bm{q}') \log  R_i(\bm{q}')
			\\
			&\geq \sum_{j\in \mathcal{J}_2} \left(\sum_{i\in \mathcal{J}_1} R_i(\bm{q}') \cdot \frac{q_{(i,j)}}{R_i(\bm{q}')}\right) \log \left(\sum_{i\in \mathcal{J}_1} R_i(\bm{q}') \cdot \frac{q_{(i,j)}}{R_i(\bm{q}')}\right)+\sum_{i\in \mathcal{J}_1}  R_i(\bm{q}') \log  R_i(\bm{q}')
			\\
			&=\sum_{j\in \mathcal{J}_2} S_j(\bm{q}')\log S_j(\bm{q}')+\sum_{i\in \mathcal{J}_1}  R_i(\bm{q}') \log  R_i(\bm{q}')
			\\
			&=RR(\bm{q})+SS(\bm{q}),
		\end{align*}
		where the third line follows from the  function $$\phi(x)=\left\{ \begin{aligned}
			&0,  & & x=0,\\
			&x\log x,\quad & & x>0,
		\end{aligned}\right.$$ is  strict concave. Also, \eqref{e5} becomes an equality if and only if for each $j\in \mathcal{J}_2$, $\frac{q_{(i,j)}}{R_i(\bm{q})}=\frac{q_{(i',j)}}{R_{i'}(\bm{q})}$ hold for all $i,i'\in \mathcal{J}_1$.
		
		Towards the last statement, as
		the ``if'' part is obvious, we only need to deal with the ``only if'' part. Noticing that $q_l>0$, for $l=1,\cdots, d$, we have $\mathcal{J}_1\times \mathcal{J}_2 =\mathcal{J}$ and ${q_{(i,j)}}/{R_i(\bm{q})}$ is independent of $i$. Let $c_j={q_{(i,j)}}/{R_i(\bm{q})}$, then $q_{(i,j)} =c_j \cdot R_i(\bm{q})$. So $S_j(\bm{q})=\sum_{i\in \mathcal{J}_1} q_{(i,j)}=c_j \sum_{i\in \mathcal{J}_1} R_i(\bm{q})=c_j$, which gives $q_{(i,j)}=R_i(\bm{q}) \cdot S_j(\bm{q})$.
	\end{proof}

	\begin{lemma}\label{le2}
		For $\bm{q}\in \emph{int}(\mathcal{S})$, we have
		$$\left \{
		\begin{aligned}
			&g_1(\bm{q})\geq g_2(\bm{q}), \quad  & & \text{if } \bm{q}\in \mathcal{S}_1 \setminus \mathcal{S}_2,
			\\
			&g_1(\bm{q})=g_2(\bm{q}), & &\text{if } \bm{q}\in \mathcal{S}_1 \cap  \mathcal{S}_2,
			\\
			&g_1(\bm{q})\leq g_2(\bm{q}), & &\text{if } \bm{q}\in \mathcal{S}_2 \setminus \mathcal{S}_1.
		\end{aligned}
		\right.
		$$
		When $\mathcal{J}_1 \times \mathcal{J}_2 \neq \mathcal{J}$, the above inequalities strictly hold.
	\end{lemma}
	\begin{proof}
		Note that 
		$$
		\begin{aligned}
			g_1(\bm{q})-g_2(\bm{q})
			&=\big(QQ(\bm{q})-RR(\bm{q})-SS(\bm{q})\big)\left(\frac{1}{SB(\bm{q})}-\frac{1}{RA(\bm{q})}\right).
		\end{aligned}
		$$
		By Lemma \ref{le1}, we have $QQ(\bm{q})-RR(\bm{q})-SS(\bm{q})\geq 0$. For $\bm{q}\in \mathcal{S}_1 \setminus \mathcal{S}_2$, $RA(\bm{q})>SB(\bm{q})$ implies $g_1(\bm{q}) \geq g_2(\bm{q})$, and other cases follow  similarly.
		\vspace{0.2cm}
		
		When $\mathcal{J}_1 \times \mathcal{J}_2 \neq \mathcal{J}$, we have $QQ(\bm{q})-RR(\bm{q})-SS(\bm{q})> 0$, so that $g_1(\bm{q})> g_2(\bm{q})$ if $\bm{q}\in \mathcal{S}_1 \setminus \mathcal{S}_2$ and $g_1(\bm{q})< g_2(\bm{q})$ if $\bm{q}\in \mathcal{S}_2 \setminus \mathcal{S}_1$.
		
	\end{proof}

	\begin{lemma}\label{le5}
		There exist $\bm{q}^{(1)}, \bm{q}^{(2)}\in \mathcal{S}$ such that \begin{equation}\label{e3}
			g_1(\bm{q}^{(1)})=G_1, \quad \quad
			g_2(\bm{q}^{(2)})=G_2.
		\end{equation}
		Furthermore, any $\bm{q}^{(1)},\bm{q}^{(2)}$ satisfying \eqref{e3} must belong to $\emph{int}(\mathcal{S})$, and have the forms:
		$$
		\bm{q}^{(1)}=\left( a_{i_l}^{\theta_1} b_{j_l}^{\lambda_1} \big( \sum_{m\in I_{i_l}} b_{j_m}^{\lambda_1} \big)^{\rho_1-1} \right)_{l=1}^d \quad \text{ and }\quad
		\bm{q}^{(2)}=\left( b_{j_l}^{\theta_2} a_{i_l}^{\lambda_2} \big( \sum_{m\in J_{j_l}} a_{i_m}^{\lambda_2} \big)^{\rho_2-1} \right)_{l=1}^d,
		$$
		where $\theta_1,\lambda_1,\rho_1$ (depending on $\bm{q}^{(1)}$) and $\theta_2, \lambda_2, \rho_2$ (depending on $\bm{q}^{(2)}$) satisfy 
		\begin{equation}\label{e93}
			\begin{aligned}
				\theta_1&=\frac{RR(\bm{q}^{(1)})}{RA(\bm{q}^{(1)})}, \quad & & & 
				\lambda_1&=\frac{QQ(\bm{q}^{(1)})-RR(\bm{q}^{(1)})}{SB(\bm{q}^{(1)})}, \quad & & &
				\rho_1&=\frac{RA(\bm{q}^{(1)})}{SB(\bm{q}^{(1)})},
				\\
				\theta_2&=\frac{SS(\bm{q}^{(2)})}{SB(\bm{q}^{(2)})}, \quad & & &
				\lambda_2&=\frac{QQ(\bm{q}^{(2)})-SS(\bm{q}^{(2)})}{RA(\bm{q}^{(2)})}, \quad & & &
				\rho_2&=\frac{SB(\bm{q}^{(2)})}{RA(\bm{q}^{(2)})}.
			\end{aligned}
		\end{equation}
	\end{lemma}
	\begin{proof}
		The proof is essentially same as that of \cite[Proposition 3.4]{LG92} by Lalley and Gatzouras via the Lagrange multipliers method, but there is a  difference in that they originally obtained that  $\bm{q}^{(1)}$ and $\bm{q}^{(2)}$ are in the forms:
		$$
		\bm{q}^{(1)}=\left(C_1 a_{i_l}^{\theta_1} b_{j_l}^{\lambda_1} \big( \sum_{m\in I_{i_l}} b_{j_m}^{\lambda_1} \big)^{\rho_1-1} \right)_{l=1}^d \text{ and }
		\bm{q}^{(2)}=\left(C_2 b_{j_l}^{\theta_2} a_{i_l}^{\lambda_2} \big( \sum_{m\in J_{j_l}} a_{i_m}^{\lambda_2} \big)^{\rho_2-1} \right)_{l=1}^d
		$$
		for some $C_1,C_2>0$. So we only need to point out that $C_1=C_2=1$.
		
		Write $\bm{q}^{(1)}= (q_1^{(1)}, \cdots q_{d}^{(1)} )$. Note that 
		$$
		\begin{aligned}
			QQ(\bm{q}^{(1)}) &= \log C_1 +\theta_1 RA(\bm{q}^{(1)})+\lambda_1 SB(\bm{q}^{(1)}) +(\rho_1-1) \sum_{l=1}^d q_l^{(1)} \log ( \sum_{m\in I_{i_l}} b_{j_m}^{\lambda_1} \big),
			\\
			RR(\bm{q}^{(1)}) &= \log C_1 + \theta_1 RA(\bm{q}^{(1)}) + \rho_1 \sum_{l=1}^d q_l^{(1)} \log ( \sum_{m\in I_{i_l}} b_{j_m}^{\lambda_1} \big).
		\end{aligned}
		$$
		By \eqref{e93}, we have 
		$$
		\begin{aligned}
			\lambda_1&=\lambda_1-\frac{1}{SB(\bm{q}^{(1)})} \sum_{l=1}^d q_l^{(1)} \log ( \sum_{m\in I_{i_l}} b_{j_m}^{\lambda_1} \big),
			\\
			\theta_1 & = \frac{\log C_1}{RA(\bm{q}^{(1)})} +\theta_1  +\frac{1}{SB(\bm{q}^{(1)})} \sum_{l=1}^d q_l^{(1)} \log ( \sum_{m\in I_{i_l}} b_{j_m}^{\lambda_1} \big),
		\end{aligned}
		$$
		which implies that $\sum_{l=1}^d q_l^{(1)} \log ( \sum_{m\in I_{i_l}} b_{j_m}^{\lambda_1} \big)=0$, then $C_1=1$. Similarly, $C_2=1$. 
	\end{proof}
	\begin{remark}
		It is worth mentioning that the above maximum points of $g_1,g_2$ may not be unique, see \cite{BF11} for a proof by Barral and Feng in the situation of Lalley and Gatzouras.
	\end{remark}
	
	The following lemma provides an expression of $\dimh K$, slightly different from that of \bara ~(Proposition \ref{prop2}).
	
	\begin{lemma}\label{le7}
		Let $K$ be a \bara~carpet. Then $$
		\dimh K=\max \{G_1, G_2\}.
		$$ 
		Moreover, when $\mathcal{J}_1 \times \mathcal{J}_2 \neq \mathcal{J}$, for any $\bm{q}^{(1)}, \bm{q}^{(2)}$ in $\mathcal{S}$ satisfying $g_1(\bm{q}^{(1)})=G_1$ and $g_2(\bm{q}^{(2)})=G_2$, we have 
		\begin{equation}\label{e64}
			\left\{
			\begin{aligned}
				&\bm{q}^{(1)} \in \mathcal{S}_1, \quad & &\text{if }G_1\geq G_2,
				\\
				&\bm{q}^{(2)} \in \mathcal{S}_2, \quad & &\text{if }G_1\leq G_2.
			\end{aligned}
			\right.
		\end{equation}
	\end{lemma}
	\begin{proof}
		Without loss of generality, assume that $G_1\geq G_2$.
		Combining \eqref{e84} and  Lemma \ref{le2}, for $\bm{q}\in \text{int}(\mathcal{S})$, we have 
		\begin{equation}\label{e63}
			g(\bm{q})=\left \{ \begin{aligned}
				&g_1(\bm{q})\geq g_2(\bm{q}),\quad & &\text{ if }\bm{q}\in \mathcal{S}_1,
				\\
				&g_2(\bm{q})\geq g_1(\bm{q}), & &\text{ if }\bm{q}\in \mathcal{S}_2 \setminus \mathcal{S}_1.
			\end{aligned}\right.
		\end{equation}
		By Lemma \ref{le5}, for any $\bm{q}^{(1)}\in \mathcal{S}$ with $g_1(\bm{q}^{(1)})=G_1$, $\bm{q}^{(1)}\in \text{int}(\mathcal{S})$. Then from Proposition \ref{prop2}, $$
		\dimh K=\max_{\bm{q}\in \mathcal{S}} g(\bm{q})\geq g(\bm{q}^{(1)}) \geq g_1(\bm{q}^{(1)})=\max \{G_1,G_2\},
		$$ where the second inequality follows from a discussion of $\bm{q}^{(1)} \in \mathcal{S}_1$ or $\mathcal{S}_2\setminus \mathcal{S}_1$. The other direction follows by a similar argument.
		
		When $\mathcal{J}_1 \times \mathcal{J}_2 \neq \mathcal{J}$, by Lemma \ref{le2}, the inequalities in \eqref{e63} strictly hold unless $\bm{q}\in \mathcal{S}_1\cap \mathcal{S}_2$. Still we only look at the case $G_1\geq G_2$. Suppose $\bm{q}^{(1)}\in \mathcal{S}_2\setminus \mathcal{S}_1$, then $G_1=g_1(\bm{q}^{(1)})<g_2(\bm{q}^{(1)})\leq G_2,$ a contradiction. Thus $\bm{q}^{(1)}\in \mathcal{S}_1$.
	\end{proof}

	In the following Proposition \ref{le8}, we will strengthen \eqref{e64} by replacing $\mathcal{S}_1, \mathcal{S}_2$ with $\text{int}(\mathcal{S}_1)$ and $\text{int}(\mathcal{S}_2)$, respectively.
	\begin{lemma}\label{le16}
		Define a function $y:\mathbb{R}\to \mathbb{R}$ by 
		$$
		\sum_{l=1}^d a_{i_l}^x b_{j_l}^{y(x)}=1.
		$$
		If $\mathcal{S}_1\cap \mathcal{S}_2\neq \emptyset$, then the function $y(x)$ satisfies  
		\begin{equation}\label{e8}
			y'(x)=-\frac{\sum_{l=1}^d a_{i_l}^x b_{j_l}^{y(x)} \log a_{i_l}}{\sum_{l=1}^d a_{i_l}^x b_{j_l}^{y(x)} \log b_{j_l}} \quad \text{ and }\quad y''(x)>0.
		\end{equation}
	\end{lemma}
	\begin{proof}
		By theorem of implicit function, we have
		\begin{equation}\label{e7}
			\sum_{l=1}^d  a_{i_l}^x b_{j_l}^{y(x)}(\log a_{i_l} +y'(x) \log b_{j_l})=0,
		\end{equation}
		which implies that the first part of \eqref{e8}.
		Differentiating \eqref{e7} implicitly with respect to $x$ gives 
		$$
		\sum_{l=1}^d  a_{i_l}^x b_{j_l}^{y(x)} \big( (\log a_{i_l} +y'(x) \log b_{j_l})^2 + y''(x) \log b_{j_l}\big)=0,
		$$
		which implies that $y''(x)> 0$ unless $\log a_{i_l}/\log b_{j_l}$ are the same for all $l=1,\cdots, d$. 
		
		Suppose that all $\log a_{i_l}/\log b_{j_l}$ are the same, set $t=\log a_{i_l} /\log b_{j_l}$. By Assumption \eqref{e38}, $t\neq 1$. Then, 
		$$
		\mathcal{S}_1=\left \{\bm{q}\in \mathcal{S}:\sum_{l=1}^d q_l \log a_{i_l}\geq \sum_{l=1}^d q_l \log b_{j_l} \right \}=\left \{ \begin{aligned}
			&\emptyset, & & \text{ if }t>1,\\
			&\mathcal{S},& & \text{ if }t<1,
		\end{aligned}\right.
		$$
		so that $\mathcal{S}_1\cap \mathcal{S}_2 =\emptyset$, a contraction. Thus, $y''(x)>0$.
		\vspace{0.2cm}
	\end{proof}
	
	\begin{lemma}\label{le6}
		For a \bara~carpet $K$, if $\mathcal{S}_1 \cap \mathcal{S}_2 \neq \emptyset$ and $\mathcal{J}_1\times \mathcal{J}_2 \neq \mathcal{J}$, then $\dimh K >\max_{\bm{q}\in  \mathcal{S}_1\cap \mathcal{S}_2} g(\bm{q})$.
	\end{lemma}
	\begin{proof}
		Note that for $\bm{q}\in \mathcal{S}_1 \cap \mathcal{S}_2$,
		\begin{equation}\label{e2}
			g(\bm{q})=g_1(\bm{q})=g_2(\bm{q})=\frac{\sum_{l=1}^{d} q_l\log q_l}{\sum_{l=1}^d q_l \log b_{j_l}} =\frac{\sum_{l=1}^{d} q_l\log q_l}{\sum_{l=1}^d q_l \log a_{i_l}}.
		\end{equation}
		We will prove the lemma by contradiction, suppose that $\dimh K=\max_{\bm{q}\in  \mathcal{S}_1\cap \mathcal{S}_2} g(\bm{q})$. A similar argument in \cite[Proposition 3.4]{LG92} implies that any maximum point of \eqref{e2} belongs to \text{int}($\mathcal{S}$). By the Lagrange multipliers method, regarding \eqref{e2} as the objective function and $\bm{q} \longmapsto RA(\bm{q})-SB(\bm{q})$ as the constraint function, we have any maximum point $\bm{q}$ of \eqref{e2} has the form:
		$$
		\bm{q}= \left(C a_{i_l}^{x_0} b_{j_l}^{y_0}\right)_{l=1}^d,
		$$
		where $x_0+y_0=g(\bm{q})$, the constant $C$ ensures $\bm{q}$ is a probability vector, and $\bm{q}\in \mathcal{S}_1 \cap \mathcal{S}_2$. Substituting $\bm{q}$ into \eqref{e2}, we further see that $C=1$. Let $y=y(x)$ be the same function in Lemma \ref{le16}, then $y(x_0)=y_0$ and $y'(x_0)=-1$. Since $\mathcal{S}_1\cap \mathcal{S}_2 \neq \emptyset$, by Lemma \ref{le16}, $y'(x)$ is strictly increasing, and so $x_0$ is the unique number satisfying $y'(x_0)=-1$.

		Fix the above $\bm{q}=\left( a_{i_l}^{x_0} b_{j_l}^{y(x_0)} \right)_{l=1}^d$. 
		For any $\epsilon>0$,  $y''(x_0)>0$ yields  $y'(x_0+\epsilon)>-1$, so $x_0+\epsilon+y(x_0+\epsilon)>x_0+y(x_0)$. Take 
		$$
		\bm{q}^{(\epsilon)}=\left( a_{i_l}^{x_0+\epsilon} b_{j_l}^{y(x_0+\epsilon)} \right)_{l=1}^d,
		$$   
		then from $y'(x_0+\epsilon)=-\frac{RA(\bm{q}^{(\epsilon)})}{SB({\bm{q}^{(\epsilon)}})}>-1$, $\bm{q}^{(\epsilon)}  \in \text{int}(\mathcal{S}_1)$.
		Therefore 
		\begin{align*}
			g(\bm{q}^{(\epsilon)})&=g_1(\bm{q}^{(\epsilon)})=\frac{QQ(\bm{q}^{(\epsilon)})}{SB(\bm{q}^{(\epsilon)})}+RR(\bm{q}^{(\epsilon)})\left(\frac{1}{RA(\bm{q}^{(\epsilon)})}-\frac{1}{SB(\bm{q}^{(\epsilon)})}\right)
			\\
			&=\frac{(x_0+\epsilon)RA(\bm{q}^{(\epsilon)})+y(x_0+\epsilon)SB(\bm{q}^{(\epsilon)})}{SB(\bm{q}^{(\epsilon)})}
			\\
			&\quad \quad +\left((x_0+\epsilon)RA(\bm{q}^{(\epsilon)})+\sum_{i=1}^r R_i(\bm{q}^{(\epsilon)})\log \sum_{l\in {I}_i} b_{j_l}^{y(x_0+\epsilon)}\right)\left(\frac{1}{RA(\bm{q}^{(\epsilon)})}-\frac{1}{SB(\bm{q}^{(\epsilon)})}\right)
			\\
			&=x_0+\epsilon+y(x_0+\epsilon)+\left(\sum_{i=1}^r R_i(\bm{q}^{(\epsilon)})\log \sum_{l\in {I}_i} b_{j_l}^{y(x_0+\epsilon)}\right)\left(\frac{1}{RA(\bm{q}^{(\epsilon)})}-\frac{1}{SB(\bm{q}^{(\epsilon)})}\right)
			\\
			&\leq x_0+y(x_0),
		\end{align*}
		where the last inequality follows from $g(\bm{q}^{(\epsilon)})\leq \max_{\bm{q}'\in  \mathcal{S}_1\cap \mathcal{S}_2} g(\bm{q}')=x_0+y(x_0)$. Combining this with $x_0+\epsilon +y(x_0+\epsilon)>x_0+y(x_0)$ and ${RA(\bm{q}^{(\epsilon)})} >{SB(\bm{q}^{(\epsilon)})}
		$  \big(since $\bm{q}^{(\epsilon)}\in \text{int}(\mathcal{S}_1)$\big), we have 
		$$
		\sum_{i=1}^r R_i(\bm{q}^{(\epsilon)})\log \sum_{l\in {I}_i} b_{j_l}^{y(x_0+\epsilon)}>0.
		$$
		Letting $\epsilon\to 0^+$, we have 
		$$
		\sum_{i=1}^r R_i(\bm{q})\log \sum_{l\in I_i} b_{j_l}^{y(x_0)}\geq 0.
		$$

		On the other hand, for $\epsilon<0$, we have  $y'(x_0+\epsilon)<-1$ and $x_0+\epsilon+y(x_0+\epsilon)>x_0+y(x_0)$. Using a similar argument as above, we have 
		$$
		\sum_{j=1}^s S_j(\bm{q}) \log \sum_{l\in J_j} a_{i_l}^{x_0}\geq 0.
		$$
		Therefore, 
		$$
		\begin{aligned}
			QQ(\bm{q})-RR(\bm{q})-SS(\bm{q})&=\sum_{l=1}^d a_{i_l}^{x_0} b_{j_l}^{y(x_0)} \log \frac{1}{(\sum_{m\in J_{j_l}} a_{i_m}^{x_0})(\sum_{m\in I_{i_l}} b_{j_m}^{y(x_0)})}
			\\
			&=-\sum_{j=1}^s S_j(\bm{q}) \log \sum_{l\in J_j} a_{i_l}^{x_0}-\sum_{i=1}^r R_i(\bm{q})\log \sum_{l\in I_i} b_{j_l}^{y(x_0)}\leq 0,
		\end{aligned}
		$$
		which contradicts with Lemma \ref{le1} since $\mathcal{J}_1\times \mathcal{J}_2 \neq \mathcal{J}$. So $\dimh K>\max_{\bm{q}\in  \mathcal{S}_1\cap \mathcal{S}_2} g(\bm{q})$.
	\end{proof}
	
	\begin{proposition}\label{le8}
		When $\mathcal{J}_1\times \mathcal{J}_2 \neq \mathcal{J}$, we have 
		\begin{enumerate}[(a)]
			\item \label{e28} if $G_1\geq G_2$, then for any $\bm{q}^{(1)}$ satisfying $g_1(\bm{q}^{(1)})=G_1$, $\bm{q}^{(1)}\in \emph{int}(\mathcal{S}_1)$,
			\item \label{e29} if $G_2\geq G_1$, then for any $\bm{q}^{(2)}$ satisfying $g_2(\bm{q}^{(2)})=G_2$, $\bm{q}^{(2)} \in \emph{int}(\mathcal{S}_2)$.
		\end{enumerate} 
	\end{proposition}
	\begin{proof}
		If $\mathcal{S}_1\cap \mathcal{S}_2 =\emptyset$, then either $\mathcal{S}_1=\mathcal{S}$ or $\mathcal{S}_2=\mathcal{S}$, since otherwise by picking $\bm{q}{'} \in \mathcal{S}_1$, $\bm{q}{''} \in \mathcal{S}_2$, an easy computation can yield that  $\alpha \bm{q}{'} +(1-\alpha) \bm{q}{''} \in \mathcal{S}_1 \cap \mathcal{S}_2$ for some $\alpha\in (0,1)$. The proposition follows by Lemma \ref{le5} immediately. For $\mathcal{S}_1\cap \mathcal{S}_2 \neq \emptyset$, we only prove the case \eqref{e28}, since \eqref{e29} is similar. For $\bm{q}^{(1)}$ satisfying $g_1(\bm{q}^{(1)})=G_1$, by Lemmas \ref{le5} and \ref{le7}, $\bm{q}^{(1)}\in \mathcal{S}_1\cap \text{int}(\mathcal{S})$. Combining this with Lemma \ref{le6}, we see that $\bm{q}^{(1)} \in \text{int}(\mathcal{S}_1)$.
	\end{proof}

	\begin{proof}[Proof of Proposition \ref{le15}]
		When $\mathcal{J}_1\times \mathcal{J}_2 \neq \mathcal{J}$,
		by Lemmas \ref{le2}, \ref{le7} and Proposition \ref{le8}, we have 
		\begin{equation}\label{e82}
			\begin{aligned}
				G_1>G_2 &\iff \dimh K=G_1\text{ and } \{\bm{q}\in \mathcal{S}:g(\bm{q})=\dimh K\}\subseteq  \text{int}(\mathcal{S}_1),
				\\
				G_1=G_2 &\iff \dimh K=G_1=G_2\text{ and }\left\{
				\begin{aligned}
					\{\bm{q}\in \mathcal{S}:g(\bm{q})=\dimh K\}\cap  \text{int}(\mathcal{S}_1)\neq \emptyset, 
					\\
					\{\bm{q}\in \mathcal{S}:g(\bm{q})=\dimh K\}\cap  \text{int}(\mathcal{S}_2)\neq \emptyset.
				\end{aligned} \right.
			\end{aligned}
		\end{equation}
		
		``$\ref{H} \Longrightarrow \ref{H'}$''. Without loss of generality, we assume that there is a $\bm{q}\in \mathcal{S}_1$ such that $g(\bm{q})=\dimh K$ and $K$ has uniform vertical fibres. If $G_1>G_2$, then $\ref{H'}$ follows; if $G_1<G_2$, this yields a contradiction since $\dimh K=g(\bm{q})=g_1(\bm{q}) \leq G_1< \max\{G_1,G_2\}=\dimh K$; if $G_1=G_2$, by \eqref{e82}, we know that there is a $\bm{q}'\in \mathcal{S}_2$ such that $g(\bm{q}')=\dimh K$, then  $K$ has both uniform vertical and horizontal fibres. So \ref{H'} holds.
		
		``$\ref{H'} \Longrightarrow \ref{H}$''. When $G_1=G_2$, \ref{H} holds obviously since  $K$ has both uniform vertical and horizontal fibres. So without loss of generality, we then only consider the case $G_1>G_2$.
		By \ref{H'}, $K$ has uniform vertical fibres, and  by \eqref{e82}, $\{\bm{q}\in \mathcal{S}_2:g(\bm{q})=\dimh K\}=\emptyset$. This still gives \ref{H}.  
		
		\vspace{0.2cm}
		
		When $\mathcal{J}_1\times \mathcal{J}_2 =\mathcal{J}$, there is nothing to prove,  since $K$ always has both uniform vertical and horizontal fibres.
	\end{proof}

	\subsection{Approximate squares}\label{sec3.2}
	Let $K$ be a \bara~carpet generated by an IFS $\{\psi_1,\cdots,\psi_d\}$. Write $\Sigma^{\mathbb{N}}=\{1,\cdots, d\}^{\mathbb{N}}$ to denote the set of all \textit{infinite words}. Write ${\Sigma}^*=\bigcup_{n\geq 1} \{1,\cdots,d\}^n$ to denote the set of all \textit{finite words}. For $\varepsilon=\varepsilon_1\cdots \varepsilon_n\in {\Sigma}^*$, write $|\varepsilon|=n$ to denote the \textit{length} of $\varepsilon$. For $\varepsilon\in {\Sigma}^\mathbb{N}$, let $|\varepsilon|=\infty$. For $\varepsilon \in {\Sigma}^* \cup {\Sigma}^\mathbb{N}$ and $m\leq |\varepsilon|$, 
	write $\varepsilon|_m=\varepsilon_1\cdots\varepsilon_m$. For $\varepsilon,\varepsilon'\in {\Sigma}^*\cup {\Sigma}^\mathbb{N}$, write $\varepsilon \prec \varepsilon'$ if $\varepsilon'|_m=\varepsilon$ for some $m\leq |\varepsilon'|$. For $\varepsilon=\varepsilon_1\cdots\varepsilon_n$, $ \varepsilon'=\varepsilon'_1\cdots\varepsilon'_m\in {\Sigma}^*$, let $\varepsilon\varepsilon'=\varepsilon_1\cdots\varepsilon_n\varepsilon'_1\cdots\varepsilon'_m$ be the \textit{concatenation} of $\varepsilon$ and $\varepsilon'$.
	For $\varepsilon=\varepsilon_1\cdots\varepsilon_n\in {\Sigma}^*$, denote  $
	[\varepsilon]=\{\varepsilon'\in {\Sigma}^\mathbb{N}:\varepsilon \prec \varepsilon'\}
	$ the \textit{cylinder set} of $\varepsilon$.
	Let $\Pi:\Sigma^\mathbb{N}\to K$ be the canonical \textit{projection} defined by
	$$
	\big\{\Pi(\varepsilon) \big\}=\bigcap_{n=1}^{\infty} \psi_{\varepsilon_1} \circ \cdots \circ \psi_{\varepsilon_n}([0,1]^2),
	$$
	for $\varepsilon=\varepsilon_1\varepsilon_2\cdots\in {\Sigma}^\mathbb{N}$. So $K=\Pi({\Sigma}^\mathbb{N})$. For $\bm{q}=(q_1,\cdots, q_d)\in \mathcal{S}$, let $\nu_{\bm{q}}$ be the Bernoulli measure on ${\Sigma}^\mathbb{N}$ associated with $\bm{q}$, and let $\mu_{\bm{q}}=\nu_{\bm{q}}\circ \Pi^{-1}$ be the self-affine measure on $K$ associated with $\bm{q}$. Clearly,
	$$
	\mu_{\bm{q}}=\sum_{l=1}^d q_l\cdot \mu_{\bm{q}}\circ\psi_l^{-1}.
	$$

	For $\varepsilon=\varepsilon_1\cdots\varepsilon_n\in {\Sigma}^*$, write $\psi_{\varepsilon}=\psi_{\varepsilon_1}\circ \cdots \circ \psi_{\varepsilon_n}$ and  
	$$
	A_{\varepsilon}=a_{i_{\varepsilon_1}} \cdots a_{i_{\varepsilon_n}}, \quad
	B_{\varepsilon}=b_{j_{\varepsilon_1}} \cdots b_{j_{\varepsilon_n}}.
	$$
	Note that $A_{\varepsilon}$, $B_{\varepsilon}$ are the width and height of the rectangle $\psi_{\varepsilon}([0,1]^2)$, respectively. For $\varepsilon=\varepsilon_1\varepsilon_2\cdots\in {\Sigma}^\mathbb{N}$ and $\delta>0$, let 
	$$
	\begin{aligned}
		k_1(\varepsilon,\delta)&=\max \{ k\in \mathbb{N}:\delta\leq A_{\varepsilon|_k}\}, 
		\\
		k_2(\varepsilon,\delta)&=\max \{ k\in \mathbb{N}:\delta\leq B_{\varepsilon|_k}\},
	\end{aligned}
	$$
	and 
	$$
	k_{\max}(\varepsilon,\delta)=\max \{k_1(\varepsilon,\delta), k_2(\varepsilon,\delta)\}, \quad 
	k_{\min}(\varepsilon,\delta)=\min \{k_1(\varepsilon,\delta), k_2(\varepsilon,\delta)\}.
	$$
	Define the \textit{approximate square} $Q(\varepsilon,\delta)$ that contains $\Pi(\varepsilon)$, with ``radius'' $\delta$ in the following way: if $k_1(\varepsilon,\delta)\geq k_2(\varepsilon,\delta)$, define 
	$$
	Q(\varepsilon,\delta):= \psi_{\varepsilon|_{k_2(\varepsilon,\delta)}}([0,1]^2) \cap  \pi_1^{-1}\circ \pi_1 \left(   \psi_{\varepsilon|_{k_1(\varepsilon,\delta)}}([0,1]^2)\right),
	$$
	if $k_2(\varepsilon,\delta)>k_1(\varepsilon,\delta)$, define
	$$
	Q(\varepsilon,\delta):= \psi_{\varepsilon|_{k_1(\varepsilon,\delta)}}([0,1]^2) \cap  \pi_2^{-1}\circ \pi_2 \left(   \psi_{\varepsilon|_{k_2(\varepsilon,\delta)}}([0,1]^2)\right).
	$$
	The following property of approximate squares follows from \cite[Lemma 7.1]{F14}.
	\begin{lemma}(\cite[Lemma 7.1]{F14})\label{le11}
		There is $C>1$ such that for any $\delta>0$, we have 
		\begin{enumerate}[(a)]
			\item for any $\varepsilon\in {\Sigma}^\mathbb{N}$, the approximate square $Q(\varepsilon,\delta)$ is a rectangle with both width and height belong to $[\delta,C\delta]$, so $
			Q(\varepsilon,\delta)\subseteq B(\Pi(\varepsilon),\sqrt{2}C\delta),
			$
			\item for any $x\in K=\Pi({\Sigma}^{\mathbb{N}})$, the ball $B(x,\delta)$ can be covered by at most $C$ approximate squares of ``radius'' $\delta$.
		\end{enumerate}
	\end{lemma}
	
	For an approximate square $Q(\varepsilon,\delta)$, let 
	$$
	U(\varepsilon,\delta)=\left\{\varepsilon'\in {\Sigma}^\mathbb{N}:\psi_{\varepsilon'|_{k_{\max}(\varepsilon,\delta)}}(K)\subseteq Q(\varepsilon,\delta) \right\}.
	$$
	Note that for $\varepsilon'\in U(\varepsilon,\delta)$, we have 
	$$
	\varepsilon|_{k_{\min}(\varepsilon,\delta)}\prec \varepsilon',
	$$
	and if $k_1(\varepsilon,\delta)\geq k_2(\varepsilon,\delta)$, then 
	$$
	i_{\varepsilon'_{k_{2}(\varepsilon,\delta)+1}}=i_{\varepsilon_{k_{2}(\varepsilon,\delta)+1}},\cdots, i_{\varepsilon'_{k_{1}(\varepsilon,\delta)}}=i_{\varepsilon_{k_{1}(\varepsilon,\delta)}},
	$$
	if $k_2(\varepsilon,\delta)\geq k_1(\varepsilon,\delta)$, then 
	$$
	j_{\varepsilon'_{k_{1}(\varepsilon,\delta)+1}}=j_{\varepsilon_{k_{1}(\varepsilon,\delta)+1}},\cdots, j_{\varepsilon'_{k_{2}(\varepsilon,\delta)}}=j_{\varepsilon_{k_{2}(\varepsilon,\delta)}}.
	$$

	For $\varepsilon\in \Sigma^{\mathbb{N}}$ and $\delta>0$, we say 
	$$\left\{
	\begin{aligned}
		&Q(\varepsilon,\delta), U(\varepsilon,\delta) \text{ are of }\textbf{1}\textbf{\textit{-type}},\quad& & \text{ if }k_1(\varepsilon,\delta)\geq k_2(\varepsilon,\delta),\\
		&Q(\varepsilon,\delta), U(\varepsilon,\delta) \text{ are of }\textbf{2}\textbf{\textit{-type}},\quad& & \text{ if }k_2(\varepsilon,\delta)\geq k_1(\varepsilon,\delta).\\
	\end{aligned}\right.
	$$
	
	\begin{remark}\label{re1}
		For a \bara~ IFS $\{\psi_1,\cdots,\psi_d\}$, pick an $i\in \mathcal{J}_1$, we may define a new IFS as follows:
		$$
		\big \{\psi_{\varepsilon} \circ \psi_{\varepsilon'}: \varepsilon\in I_i, \varepsilon' \in \{1,\cdots,d\} \big \} \bigcup \big \{\psi_{\varepsilon}:\varepsilon\in \{1,\cdots,d\} \setminus I_i\big \}.
		$$
		Clearly, this IFS generates a same attractor $K$ as that of $\{\psi_1,\cdots, \psi_d\}$. According to this new IFS, the concept of $1$-type approximate squares is still well-defined.
	\end{remark}
	
	\begin{lemma}\label{le17}
		For a Borel probability measure $\nu$ on ${\Sigma}^{\mathbb{N}}$, let $\mu=\nu\circ\Pi^{-1}$.  If there is a constant $C>0$ such that for all $\varepsilon\in {\Sigma}^*$ and $l\in \{1,\cdots,d\}$, such that 
		$$
		\frac{\nu([\varepsilon l])}{\nu([\varepsilon])}\geq C.
		$$
		Then for all $\varepsilon\in {\Sigma}^*$, we have 
		\begin{equation}
			\mu\big (\psi_{\varepsilon}([0,1]^2)\big )=\nu([\varepsilon]). \nonumber
		\end{equation}
		
	\end{lemma}
	\begin{proof}
		It is similar to \cite[Lemma 7]{FW05} due to Feng and Wang. Here we provide a proof for completeness.  Noticing that for any $\varepsilon\in \Sigma^*$,  $\Pi([\varepsilon])\subseteq \psi_{\varepsilon}([0,1]^2)$, it suffices to prove that 
		\begin{equation}\label{e46}
			\mu\big (\psi_{\varepsilon}([0,1]^2)\big )\leq \nu([\varepsilon]).
		\end{equation}
		First we consider the case that there is  $l\in \{1,\cdots,d\}$ such that $\psi_l([0,1]^2)\subseteq (0,1)^2$. Let $n=|\varepsilon|$. For $k\geq n$, let 
		$$
		\mathcal{B}_{k}=\{\varepsilon'\in {\Sigma}^*: |\varepsilon'|=k, \psi_{\varepsilon'}([0,1]^2)\cap \psi_{\varepsilon}([0,1]^2)\neq \emptyset, \varepsilon \nprec \varepsilon'\}.
		$$
		For $\varepsilon'=\varepsilon'_1\cdots\varepsilon'_k\in \mathcal{B}_k$, note that $\varepsilon'_m\neq l$ for all $n<m\leq k$. Therefore 
		\begin{align*}
			\mu\big (\psi_{\varepsilon}([0,1]^2)\big )&\leq \nu([\varepsilon]) +\sum_{\varepsilon'\in \mathcal{B}_k}\nu([\varepsilon'])  \leq \nu([\varepsilon]) + \sum_{\varepsilon'\in \mathcal{B}_k}\nu([\varepsilon'|_n]) \cdot \frac{\nu([\varepsilon'|_{n+1}])}{\nu([\varepsilon'|_n])}  \cdots  \frac{\nu([\varepsilon'|_k])}{\nu([\varepsilon'|_{k-1}])}
			\\
			&\leq \nu([\varepsilon]) + \sum_{\varepsilon'_1\cdots\varepsilon'_n\in {\Sigma}^*}\nu([\varepsilon'_1\cdots\varepsilon'_n]) \cdot \prod_{m=n+1}^k \sum_{\varepsilon_{m}'\neq l} \frac{\nu([\varepsilon'_1\cdots\varepsilon'_{m}])}{\nu([\varepsilon'_1\cdots\varepsilon'_{m-1}])} 
			\\
			&\leq \nu([\varepsilon])+(1-C)^{k-n},
		\end{align*}
		which implies that $\mu\big (\psi_{\varepsilon}([0,1]^2)\big )\leq \nu([\varepsilon])$ by letting $k\to \infty$. 
		
		When for all $l\in \{1,\cdots,d\}$, $\psi_l([0,1])\setminus (0,1)^2\neq \emptyset$, by Assumption \eqref{e38}, there exists $\varepsilon=\varepsilon_1\varepsilon_2\in \Sigma^*$, such that $\psi_{\varepsilon}([0,1]^2)\subseteq (0,1)^2$. Using a similar argument with respect to IFS $\{\psi_{\varepsilon}:\varepsilon=\varepsilon_1\varepsilon_2\in \Sigma^*\}$, we know that $\eqref{e46}$ follows for all $\varepsilon\in \bigcup_{n\in \mathbb{N}}\{1,\cdots,d\}^{2n}$. For $\varepsilon\in \Sigma^*$ such that $|\varepsilon|$ is odd, noticing that 
		$$
		\mu\big (\psi_{\varepsilon}([0,1]^2) \big ) \leq \sum_{l=1}^d \mu \big (\psi_{\varepsilon l}([0,1]^2) \big ) =\sum_{l=1}^d \nu([\varepsilon l])=\nu([\varepsilon]),
		$$
		we still have \eqref{e46}.
	\end{proof}
	
	\begin{lemma}\label{le12}
		For $\bm{q}\in \emph{int}(\mathcal{S})$, let $\nu_{\bm{q}}$ and $\mu_{\bm{q}}$ be the Bernoulli and self-affine measures associated with $\bm{q}$, respectively. For $\varepsilon=\varepsilon_1\varepsilon_2\cdots\in {\Sigma}^\mathbb{N}$ and $\delta>0$, write $k_{\max}=k_{\max}(\varepsilon,\delta)$ and $k_{\min}=k_{\min}(\varepsilon,\delta)$, we have 
		$$
		\mu_{\bm{q}}\big(Q(\varepsilon,\delta)\big)=q_{\varepsilon_1}\cdots q_{\varepsilon_{k_{\min}}} \left\{
		\begin{aligned}
			&R_{i_{\varepsilon_{k_{\min}+1}}}(\bm{q})\cdots R_{i_{\varepsilon_{k_{\max}}}}(\bm{q}),  & & \text{ if } Q(\varepsilon,\delta) \text{ if of }1\text{-type}, 
			\\
			&S_{i_{\varepsilon_{k_{\min}+1}}}(\bm{q})\cdots S_{i_{\varepsilon_{k_{\max}}}}(\bm{q}),  & & \text{ if } Q(\varepsilon,\delta) \text{ if of }2\text{-type}.
		\end{aligned}\right.
		$$
	\end{lemma}
	\begin{proof}
		This follows from Lemma \ref{le17}.
	\end{proof}

	For $\varepsilon=\varepsilon_1\cdots\varepsilon_n\in {\Sigma}^{*}$, let $L_{\varepsilon}=\max \{A_{\varepsilon},B_{\varepsilon}\}$. For $\varepsilon\in {\Sigma}^{\mathbb{N}}$ and $n\in \mathbb{N}$, write $$Q(\varepsilon,n):=Q(\varepsilon,L_{\varepsilon|_n}), \quad\quad U(\varepsilon,n):=U(\varepsilon,L_{\varepsilon|_n}).$$
	
	\begin{lemma}\label{le21}
		For $\varepsilon\in \Sigma^{\mathbb{N}}$ and $n\geq 1$, then 
		\begin{enumerate}[(a)]
			\item \label{e110} $L_{\varepsilon|_n}=A_{\varepsilon|_n}$ if and only if $Q(\varepsilon,n)$ is of $1$-type,
			\item $L_{\varepsilon|_n}=B_{\varepsilon|_n}$ if and only if $Q(\varepsilon,n)$ is of $2$-type.
		\end{enumerate}
	\end{lemma}
	\begin{proof}
		It suffices to prove part \eqref{e110}. Write $\delta=L_{\varepsilon|_n}$, $k_1=k_1(\varepsilon,\delta)$ and $k_2=k_2(\varepsilon,\delta)$ for short. Note that 
		\begin{equation}\label{e36}
			A_{\varepsilon|_{k_1+1}}<\delta=\max\{A_{\varepsilon|_n}, B_{\varepsilon|_n}\}\leq A_{\varepsilon|_{k_1}}, \quad
			B_{\varepsilon|_{k_2+1}}<\delta=\max\{A_{\varepsilon|_n}, B_{\varepsilon|_n}\}\leq B_{\varepsilon|_{k_2}}.
		\end{equation}
		When $Q(\varepsilon,n)$ is of $1$-type, 
		if $A_{\varepsilon|_n}< B_{\varepsilon|_n}$, by the right of \eqref{e36}, $k_1\geq k_2=n$, so the left of \eqref{e36} gives $B_{\varepsilon|_n}\leq A_{\varepsilon|_{k_1}} \leq A_{\varepsilon|_n}$, a contradiction. Therefore $A_{\varepsilon|_n}\geq  B_{\varepsilon|_n}$ and $L_{\varepsilon|_n}=A_{\varepsilon|_n}$. Conversely, when $L_{\varepsilon|_n}=A_{\varepsilon|_n}$, by the left of \eqref{e36}, $k_1=n$, so the right of \eqref{e36} gives $B_{\varepsilon|_n}\leq A_{\varepsilon|_n}\leq B_{\varepsilon|_{k_2}}$. Then $k_2\leq n=k_1$, thus $Q({\varepsilon,n})$ is of $1$-type.
	\end{proof}

	\section{Proof of ``$\Longrightarrow$'' in Theorem \ref{th1}-\eqref{e41}  }\label{sec4}
	In this section, we prove that if  \ref{H} does not hold, then $\mathcal{H}^{\dimh K}(K)=+\infty$. 
	
	Let $\varphi:[0,+\infty) \to [0,+\infty)$ be an increasing and continuous function with $\varphi(0)=0$, called a \textit{gauge function}. The $\varphi$\textit{-Hausdorff measure} $\mathcal{H}^{\varphi}(K)$ for $K\subseteq \mathbb{R}^2$ with respect to $\varphi$ is defined as
	$$
	\mathcal{H}^{\varphi}(K)=\lim_{\delta\to0^+} \inf \left \{\sum_{l=1}^\infty \varphi(|U_l|):\{U_l\}_{l\geq 1} \text{ is a }\delta\text{-cover of }K \right\}.
	$$
	Taking $\varphi(t)=t^s$ results in the common $s$-dimensional Hausdorff measure $\mathcal{H}^s$. We will prove that $\mathcal{H}^{\varphi_c}(K)=+\infty$ holds for small $c>0$, where $\varphi_c$ is in the form:
	\begin{equation}\label{e90}
		\varphi_c(t)=t^{\dimh K} \exp\left({-c \frac{|\log t|}{(\log |\log t|)^2}}\right),
	\end{equation}
	which immediately yields that $\mathcal{H}^{\dimh K}(K)=+\infty$. 
	
	We state the following proposition adjusted from the  Rogers-Taylor density theorem \cite{RT61} for \bara~carpets.  
	
	\begin{proposition}\label{prop4}
		Let $\tilde{\Sigma}$ be a Borel subset of $\Sigma^{\mathbb{N}}$ and $\mu$ be a finite Borel measure on $\mathbb{R}^2$ such that $\mu\big(\Pi(\tilde{\Sigma})\big)>0$. Then for a gauge function $\varphi$, if for all $\varepsilon\in \tilde{\Sigma}$,  
		$$
		\liminf_{n \to \infty} \frac{\varphi(L_{\varepsilon|_n})}{\mu \big(Q(\varepsilon,n) \big)}=+\infty,
		$$
		we have
		$
		\mathcal{H}^{\varphi}\big(\Pi(\tilde{\Sigma})\big)=+\infty.
		$
	\end{proposition}
	
	Define 
	\begin{equation}\label{e65}
		\begin{aligned}
			\Sigma_1&=\left\{\varepsilon\in {\Sigma}^\mathbb{N}:Q(\varepsilon,n) \text{ is of $1$-type for infinitely many } n\in \mathbb{N}  \right\},
			\\
			\Sigma_2&=\left\{\varepsilon\in {\Sigma}^\mathbb{N}:Q(\varepsilon,n) \text{ is of $2$-type for infinitely many } n\in \mathbb{N} \right\}.
		\end{aligned}
	\end{equation}
	Let $K_1=\Pi(\Sigma_1)$ and $K_2=\Pi(\Sigma_2)$. Note that ${\Sigma}^{\mathbb{N}}=\Sigma_1\cup \Sigma_2$ and $K=K_1\cup K_2$. 
	
	Recall that in Lemma \ref{le7}, we have shown $\dimh K=\max \{G_1,G_2\}$. The following lemma is inspired by Peres's  work on Bedford-McMullen carpets \cite[Proposition 4]{P94}, but our situation is more complicated.
	
	\begin{lemma}\label{le18}
		Let $K$ be a  \bara~carpet. We have for some small $c>0$,
		\begin{enumerate}[(a)]
			\item \label{e61} if $\dimh K=G_1$ and $K$ does not have uniform vertical fibres, then $\mathcal{H}^{\varphi_c}(K_1)=+\infty$,
			\item \label{e62} if $\dimh K=G_2$ and $K$ does not have uniform horizontal fibres, then $\mathcal{H}^{\varphi_c}(K_2)=+\infty$,
		\end{enumerate}
		where $\varphi_c$ is defined in \eqref{e90}. 
	\end{lemma}
	
	\begin{proof}
		We only prove \eqref{e61} here, since \eqref{e62} follows in a similar way.
		\vspace{0.2cm}
		
		Let $\bm{q}=(q_1,\cdots,q_d)\in \mathcal{S}$ such that $g_1(\bm{q})=G_1=\dimh K$. Since $K$ does not have uniform vertical fibres, $\mathcal{J}_1 \times \mathcal{J}_2 \neq \mathcal{J}$. Using Lemma \ref{le5} and Proposition \ref{le8}, we have $\bm{q}\in \text{int}(\mathcal{S}_1)$ and $\bm{q}$ satisfies the following form:
		$$
		\bm{q}=\left( a_{i_l}^{\theta_1} b_{j_l}^{\lambda_1} \big( \sum_{m\in I_{i_l}} b_{j_m}^{\lambda_1} \big)^{\rho_1-1} \right)_{l=1}^d, \quad \quad R_i(\bm{q}) =\sum_{l\in I_i} q_l= a_i^{\theta_1} \big( \sum_{m\in I_{i}} b_{j_m}^{\lambda_1} \big)^{\rho_1},\quad \text{ for }i\in \mathcal{J}_1,
		$$
		where 
		\begin{equation}\label{e58}
			\theta_1=\frac{RR(\bm{q})}{RA(\bm{q})}, \quad 
			\lambda_1=\frac{QQ(\bm{q})-RR(\bm{q})}{SB(\bm{q})}, \quad
			\rho_1=\frac{RA(\bm{q})}{SB(\bm{q})}.
		\end{equation}
		For $l=1,\cdots,d$, let $\gamma_l=\sum_{m\in I_{i_l}} b_{j_m}^{\lambda_1}$. 
		Let $$
		\begin{aligned}
			\bm{\alpha}&=(\log a_{i_1},\log a_{i_2},\cdots,\log a_{i_d}),\\
			\bm{\beta}&=(\log b_{j_1},\log b_{j_2},\cdots,\log b_{j_d}),\\
			\bm{\gamma}&=(\log\gamma_1,\log \gamma_2,\cdots,\log \gamma_d).
		\end{aligned}
		$$
		be three vectors in $\mathbb{R}^d$. Let $\langle \cdot,\cdot \rangle$ be the standard inner product in Euclidean spaces. 
		Note that by \eqref{e83}
		$$
		RA(\bm{q})=\langle \bm{q},\bm{\alpha} \rangle,\quad
		SB(\bm{q})=\langle \bm{q},\bm{\beta} \rangle.
		$$
		Since $QQ(\bm{q})=\theta_1 RA(\bm{q})+\lambda_1 SB(\bm{q}) +(\rho_1-1) \langle \bm{q},\bm{\gamma} \rangle$ and $RR(\bm{q})=\theta_1 RA(\bm{q})+\rho_1 \langle \bm{q},\bm{\gamma} \rangle$, by \eqref{e58}, we see that 
		$$
		\lambda_1=\frac{QQ(\bm{q})-RR(\bm{q})}{SB(\bm{q})}=\frac{\lambda_1 SB(\bm{q})- \langle \bm{q},\bm{\gamma} \rangle }{SB(\bm{q})}
		$$
		which implies that $\langle \bm{q},\bm{\gamma} \rangle=0$. Note that $\bm{\gamma}$ is a non-zero vector since $K$ does not have uniform vertical fibres. Let $\bm{e}=(1,\cdots,1)\in \mathbb{R}^d$. We divide the proof into  three cases.
		\vspace{0.2cm}
		
		\textit{\textbf{Case 1.}} \textit{When $\text{rank}(\bm{\alpha},\bm{\beta},\bm{e},\bm{\gamma})=\text{rank}(\bm{\alpha},\bm{\beta},\bm{e})+1$.}
		\vspace{0.2cm}
		
		Note that the following equations have solutions:
		\begin{equation}\label{e43}
			\langle \bm{\alpha},\bm{u} \rangle =0,\quad 
			\langle \bm{\beta},\bm{u} \rangle =0,\quad
			\langle \bm{e},\bm{u} \rangle =0,\quad
			\langle \bm{\gamma},\bm{u} \rangle \neq 0.
		\end{equation}
		Choose a vector  $\bm{u}=(u_1,\cdots,u_d)$ satisfying \eqref{e43} and 
		$\Delta:=\langle \bm{\gamma},\bm{u} \rangle>0$. Fix a small $\delta>0$. Note that $\langle \bm{u}, \bm{e} \rangle=0$. For each $n\geq 2$, defined a probability vector $\bm{q}^{(n)}$ by
		$$
		\bm{q}^{(n)}=\left(1-\frac{\delta}{\log n}\right) \bm{q}  +\frac{\delta}{\log n} \left(\bm{q}+\bm{u} \right)=\bm{q}+ \frac{\delta}{\log n} \bm{u}.
		$$ 
		Let $\bm{q}^{(1)}=\bm{q}$. Denote $\nu_\delta$ to be the product measure $\prod_{n=1}^{+\infty} \bm{q}^{(n)}$ on ${\Sigma}^{\mathbb{N}}$, and let $\mu_{\delta}=\nu_{\delta}\circ \Pi^{-1}$ be the probability measure on $K$. 
		
		For $\varepsilon=\varepsilon_1\varepsilon_2\cdots\in {\Sigma}^{\mathbb{N}}$ and $n\in \mathbb{N}$. Using Lemma \ref{le17}, we have 
		\begin{equation}\label{e44}
			\begin{aligned}
				\log \mu_{\delta} \big (\psi_{\varepsilon|_n}([0,1]^2) \big ) =\log \nu_{\delta}([\varepsilon|_n])&= \sum_{l=1}^n \log q_{\varepsilon_l}^{(l)},
				\\
				\log \mu_{\delta} \big (\pi_1^{-1}\circ \pi_1\circ \psi_{\varepsilon|_n}([0,1]^2) \big ) &= \sum_{l=1}^n \log  R_{i_{\varepsilon_l}} (\bm{q}^{(l)}).
			\end{aligned}
		\end{equation}
		Note that 
		\begin{equation}\label{e45}
			\log A_{\varepsilon|_n}=\sum_{l=1}^n \log a_{i_{\varepsilon_l}}, \quad \log B_{\varepsilon|_n}=\sum_{l=1}^{n} \log b_{j_{\varepsilon_l}}.
		\end{equation}
		Regard the right side of \eqref{e44} and \eqref{e45} as the sums of independent random variables, by the law of iterated logarithm (cf. \cite[Section 7.3.3]{BP17}), we have 
		\begin{equation}\label{e47}
			\begin{aligned}
				\Big \lvert \sum_{l=1}^{n} \log q^{(l)}_{\varepsilon_l}-\sum_{l=1}^{n} \sum_{m=1}^d q^{(l)}_m \log q^{(l)}_m \Big \rvert &=O\big( \sqrt{n\log \log n} \big),
				\\
				\Big \lvert\sum_{l=1}^n \log R_{i_{\varepsilon_l}}(\bm{q}^{(l)}) -\sum_{l=1}^n \sum_{i\in \mathcal{J}_1} R_i(\bm{q}^{(l)}) \log R_i(\bm{q}^{(l)}) \Big \rvert &=O\big( \sqrt{n\log \log n} \big),
				\\
				\Big \lvert \sum_{l=1}^n \log a_{i_{\varepsilon_l}} -\sum_{l=1}^n \sum_{i\in \mathcal{J}_1} R_i(\bm{q}^{(l)}) \log a_i \Big \rvert &=O\big( \sqrt{n\log \log n} \big),
				\\
				\Big \lvert \sum_{l=1}^n \log b_{j_{\varepsilon_l}} -\sum_{l=1}^n \sum_{j\in \mathcal{J}_2} S_j(\bm{q}^{(l)}) \log b_j  \Big \rvert &=O\big( \sqrt{n\log \log n} \big),
			\end{aligned}
		\end{equation}
		for $\nu_\delta$ a.e. $\varepsilon\in {\Sigma}^{\mathbb{N}}$. Since $\bm{u}$ satisfies \eqref{e43}, we have 
		\begin{equation}\label{e55}
			\begin{aligned}
				RA(\bm{q}^{(l)})&=\sum_{m=1}^d q^{(l)}_m \log a_{i_m}=\langle \bm{q}+\frac{\delta}{\log l} \bm{u},\bm{\alpha} \rangle=RA(\bm{q}),
				\\
				SB(\bm{q}^{(l)})&=\sum_{m=1}^d q^{(l)}_m \log b_{j_m}=\langle \bm{q}+\frac{\delta}{\log l} \bm{u},\bm{\beta} \rangle=SB(\bm{q}).
			\end{aligned}
		\end{equation}

		It follows from $\bm{q}\in \text{int}(\mathcal{S}_1)$, $RA(\bm{q})>SB(\bm{q})$ and $0<\rho_1<1$. Thus by \eqref{e47}, for $\nu_{\delta}$ a.e. $\varepsilon\in \Sigma^{\mathbb{N}}$, $A_{\varepsilon|_n}$ is much greater than $B_{\varepsilon|_n}$ for all sufficiently large $n$. So that the approximate square $Q(\varepsilon,n)$ is of $1$-type and $\nu_{\delta}(\Sigma_1)=1$.
		
		Fix $\varepsilon\in \Sigma_1$ satisfying \eqref{e47}, and sufficiently large $n$. Write $k^{(n)}=k_{\min}(\varepsilon,L_{\varepsilon|_n})$. By Lemma \ref{le21},  $L_{\varepsilon|_n}=A_{\varepsilon|_n}$, and so $k_1(\varepsilon,L_{\varepsilon|_n})=n$,  $k^{(n)}=k_2(\varepsilon,L_{\varepsilon|_n}\leq n$. 
		It follows from \eqref{e47} and \eqref{e55}, we have 
		\begin{equation}\label{e95}
			\begin{aligned}
				\big|\log A_{\varepsilon|n} -n RA(\bm{q})\big|&=O\big( \sqrt{n\log \log n} \big),
				\\ 
				\big|\log B_{\varepsilon|_{k^{(n)}}}-k^{(n)}  SB(\bm{q})\big |&=O\big( \sqrt{n\log \log n} \big).
			\end{aligned}
		\end{equation}
		By Lemma \ref{le11},  $A_{\varepsilon|_n}\leq B_{\varepsilon|_{k^{(n)}}} \leq CA_{\varepsilon|_n}$. Combining this with \eqref{e95}, we have 
		\begin{equation}\label{e96}
			\big|k^{(n)}-\lfloor n\rho_1 \rfloor\big|=O\big( \sqrt{n\log \log n} \big),
		\end{equation}
		where $\lfloor x \rfloor$ denotes the largest integer less than $x$. 
		
		Now we estimate $\log \mu_{\delta} \big (Q(\varepsilon,n) \big )$. Similar to Lemma \ref{le12}, 
		\begin{equation}\label{e98}
			\log \mu_{\delta}\big (Q(\varepsilon,n) \big )=\sum_{l=1}^{k^{(n)}} \log q_{\varepsilon_l}^{(l)} +\sum_{l=k^{(n)}+1}^n \log R_{i_{\varepsilon_l}}(\bm{q}^{(l)}).
		\end{equation} 
		Noticing that for all $l,m=1,\cdots,d$ and $i\in \mathcal{J}_1$, $\log q_{m}^{(l)}, \log R_i({\bm{q}^{(l)}})$ are bounded away from $0$ and $-\infty$, combining \eqref{e47}, \eqref{e96} and \eqref{e98}, we have 
		\begin{equation}\label{e48}
			\begin{aligned}
				\log \mu_{\delta}\big( Q(\varepsilon,n) \big) &\leq \sum_{l=1}^{\lfloor n\rho_1 \rfloor} \log q_{\varepsilon_l}^{(l)} +\sum_{l=\lfloor n\rho_1 \rfloor+1}^n \log R_{i_{\varepsilon_l}}(\bm{q}^{(l)}) +O\big( \sqrt{n\log \log n} \big)
				\\
				&\leq  \sum_{l=1}^{\lfloor n\rho_1 \rfloor} QQ(\bm{q}^{(l)}) +\sum_{l=\lfloor n\rho_1 \rfloor+1}^n RR(\bm{q}^{(l)}) +O\big( \sqrt{n\log \log n} \big).
			\end{aligned}
		\end{equation}
		Since for two probability vectors $\bm{p}=(p_1,\cdots,p_d), \bm{v}=(v_1,\cdots,v_d)$ in $\mathbb{R}^d$, for $\epsilon>0$,
		\begin{equation}\label{e49}
			QQ\big( (1-\epsilon) \bm{p}+\epsilon \bm{v} \big)=QQ(\bm{p})+ \epsilon \sum_{l=1}^d (v_l-p_l) \log p_l +O(\epsilon^2),
		\end{equation}
		by \eqref{e43}, we have 
		\begin{equation}\label{e50}
			\begin{aligned}
				QQ(\bm{q}^{(l)}) &= QQ(\bm{q}) + \frac{\delta}{\log l} (\rho_1-1) \Delta +O\left(\frac{\delta^2}{(\log l)^2}\right),
				\\
				RR(\bm{q}^{(l)})&= RR(\bm{q}) +\frac{\delta}{\log l} \rho_1 \Delta +O\left(\frac{\delta^2}{(\log l)^2}\right).
			\end{aligned} 
		\end{equation}
		Note that 
		\begin{equation}\label{e51}
			\begin{aligned}
				(\rho_1-1)\sum_{l=2}^{\lfloor n\rho_1 \rfloor} \frac{1}{\log l}+ \rho_1 \sum_{l=\lfloor n\rho_1 \rfloor+1}^n \frac{1}{\log l} &\leq   \rho_1 \int_2^{n} \frac{1}{\log t}dt- \int_{2}^{n\rho_1} \frac{1}{\log t}dt+O(1) 
				\\
				&\leq   \rho_1 \log \rho_1 \int_{2\rho_1^{-1}}^n \frac{1}{\log t\log(\rho_1 t)}dt +O(1)
				\\
				&\leq   \rho_1 \log \rho_1 \frac{n}{(\log n)^2}+O(1),
			\end{aligned}
		\end{equation}
		and 
		\begin{equation}\label{e85}
			\sum_{l=2}^n \frac{1}{(\log l)^2}= \int_2^n \frac{1}{(\log t)^2} dt+O(1) =O\left( \frac{n}{(\log n)^2} \right).
		\end{equation}
		Combining \eqref{e48}, and \eqref{e50}-\eqref{e85}, we see that 
		\begin{equation}\label{e86}
			\begin{aligned}
				\log \mu_{\delta}\big( Q(\varepsilon,n) \big) &\leq  \lfloor n\rho_1 \rfloor QQ(\bm{q}) +\sum_{l=2}^{\lfloor n\rho_1 \rfloor}\frac{\delta}{\log l}(\rho_1-1) \Delta  
				\\
				& \quad +(n-\lfloor n\rho_1 \rfloor ) RR(\bm{q})+\sum_{l=\lfloor n\rho_1 \rfloor +1}^n \frac{\delta}{\log l} \rho_1 \Delta +O\left( \sum_{l=2}^n \frac{\delta^2}{(\log l)^2} \right) 
				\\
				&\leq \lfloor n\rho_1 \rfloor QQ(\bm{q}) +(n-\lfloor n\rho_1 \rfloor) RR(\bm{q}) +\delta \Delta \rho_1 \log \rho_1 \frac{n}{(\log n)^2} +O\left( \frac{\delta^2n}{(\log n)^2} \right).
			\end{aligned}
		\end{equation}

		On the other hand, noticing that  
		$$
		\dimh K=g_1(\bm{q})=\frac{RR(\bm{q})}{RA(\bm{q})}+\frac{QQ(\bm{q})-RR(\bm{q})}{SB(\bm{q})},
		$$
		by Lemma \ref{le21} and \eqref{e95}, we have 
		\begin{equation}\label{e57}
			\begin{aligned}
				\log \varphi_c(L_{\varepsilon|_n})&=(\dimh K) \log  A_{\varepsilon|_n}-c\frac{|\log A_{\varepsilon|_n}|}{\big(\log( |\log A_{\varepsilon|_n}|)\big)^2}
				\\
				&\geq n RR(\bm{q})+n\rho_1 \big( QQ(\bm{q})-RR(\bm{q}) \big) - c\frac{ n|RA(\bm{q})|}{(\log n)^2} +O \left( \frac{n}{(\log n)^3}\right).
			\end{aligned}
		\end{equation}
		Combining this with \eqref{e86}, we have 
		\begin{equation}\label{e97}
			\begin{aligned}
				&\log \mu_{\delta} \big( Q(\varepsilon,n) \big) -\log \varphi_c(L_{\varepsilon|_n})
				\\
				\leq  &\lfloor n\rho_1 \rfloor QQ(\bm{q}) +(n-\lfloor n\rho_1 \rfloor) RR(\bm{q}) -n RR(\bm{q})-n\rho_1 \big (QQ(\bm{q})-RR(\bm{q}) \big)
				\\
				&+c\frac{ n|RA(\bm{q})|}{(\log n)^2}+\delta \Delta \rho_1 \log \rho_1 \frac{n}{(\log n)^2} +O\left( \frac{\delta^2n}{(\log n)^2} \right) \to-\infty,
			\end{aligned}
		\end{equation}
		as $n$ tends to $+\infty$ when $\delta>0$ is small and $c>0$ is much smaller than $\delta$, noticing that  $\Delta>0$, $\log \rho_1 <0$. 
		
		Let 
		$$
		\Sigma_1'=\{\varepsilon\in \Sigma_1:\lim_{n\to\infty} \frac{\varphi_c(L_{\varepsilon|_n})}{\mu_{\delta}\big(Q(\varepsilon,n)\big)}=+\infty\}.
		$$
		Then by \eqref{e97}, $\nu_{\delta}(\Sigma_1')=1$.
		Using Proposition \ref{prop4}, $\mathcal{H}^{\varphi_c}(K_1)\geq \mathcal{H}^{\varphi_c}\big(\Pi(\Sigma_1') \big)=+\infty$.
		
		\vspace{0.2cm}
		\textit{\textbf{Case 2.}} \textit{When $\text{rank}(\bm{\alpha},\bm{\beta},\bm{e},\bm{\gamma})=\text{rank}(\bm{\alpha},\bm{\beta},\bm{e})$, and for any $a,b,e$ in $\mathbb{R}$ satisfying   $\bm{\gamma}=a \bm{\alpha} +b \bm{\beta}+e \bm{e}$, $e\neq 0$.}
		\vspace{0.2cm}
		
		In this case, $\text{rank}(\bm{\alpha},\bm{\beta},\bm{e}) =\text{rank}(\bm{\alpha},\bm{\beta})+1$. Without loss of generality, assume that $I_1\neq \emptyset$ and $I_1=\{1,\cdots,v\}$ for some $v< d$. Let 
		$$
		\Lambda=I_1\times \{1,\cdots, d\} \cup \{v+1,\cdots, d\}\subseteq \Sigma^*.
		$$
		Consider the IFS $\{\psi_{\varepsilon}:\varepsilon\in \Lambda\}$. Write 
		$$
		\begin{aligned}
			\bm{\alpha}_{\Lambda}&=(\log a_{i_1}+\log a_{i_1},\cdots,\log a_{i_v}+\log a_{i_d},\log a_{i_{v+1}}, \cdots, \log a_{i_d}),\\
			\bm{\beta}_{\Lambda}&=(\log b_{j_1}+\log b_{j_1},\cdots, \log b_{j_v}+\log b_{j_d}, \log b_{j_{v+1}},\cdots,\log b_{j_d}),\\
			\bm{\gamma}_{\Lambda}&=(\log\gamma_1+\log \gamma_1,\cdots, \log \gamma_v+\log \gamma_d,\log \gamma_{v+1},\cdots,\log \gamma_d),\\
			\bm{q}_{\Lambda}&=(q_1q_1,\cdots, q_v q_d, q_{v+1},\cdots, q_d),
		\end{aligned}
		$$
		and $\bm{e}_{\Lambda}=(\overbrace{2,\cdots,2}^{vd},\overbrace{1,\cdots,1}^{d-v}),\tilde{\bm{e}}=(1,\cdots,1)$ in $\mathbb{R}^{vd+d-v}$. Let $a,b,e\in \mathbb{R}$ satisfy $\bm{\gamma}=a \bm{\alpha} +b \bm{\beta}+e \bm{e}$ $(e\neq 0)$. It is easy to check that $\bm{\gamma}_{\Lambda}=a \bm{\alpha}_{\Lambda}+b \bm{\beta}_{\Lambda}+ e \bm{e}_{\Lambda}.$ Taking elementary column transformations of vectors $\bm{\alpha}_{\Lambda}, \bm{\beta}_{\Lambda},\bm{e}_{\Lambda}, \tilde{\bm{e}}$ as follows:
		\begin{align*}
			&\begin{pmatrix}
				\log a_{i_1}+\log a_{i_1} & \cdots &\log a_{i_1}+\log a_{i_{d}} &\cdots &\log a_{i_{v+1}}&\cdots &\log a_{i_d}\\
				\log b_{j_1}+\log b_{j_1} & \cdots &\log b_{j_1}+\log b_{j_{d}}&\cdots &\log b_{j_{v+1}} &\cdots &\log b_{j_d}\\
				2&\cdots & 2 & \cdots &1 &\cdots &1 \\
				1&\cdots & 1 & \cdots &1 &\cdots &1\\
			\end{pmatrix}
			\\
			\Longrightarrow & 
			\begin{pmatrix}
				\log a_{i_1}+\log a_{i_1} & \cdots &\log a_{i_1} &\cdots &\log a_{i_{v+1}}&\cdots &\log a_{i_d}\\
				\log b_{j_1}+\log b_{j_1} & \cdots &\log b_{j_1}&\cdots &\log b_{j_{v+1}} &\cdots &\log b_{j_d}\\
				2&\cdots & 1 & \cdots &1 &\cdots &1 \\
				1&\cdots & 0 & \cdots &1 &\cdots &1\\
			\end{pmatrix}
			\\
			\Longrightarrow & 
			\begin{pmatrix}
				0 & \cdots &\log a_{i_1} &\cdots &\log a_{i_{v+1}}&\cdots &\log a_{i_d}\\
				0 & \cdots &\log b_{j_1}&\cdots &\log b_{j_{v+1}} &\cdots &\log b_{j_d}\\
				0&\cdots & 1 & \cdots &1 &\cdots &1 \\
				1&\cdots & 0 & \cdots &0 &\cdots &0\\
			\end{pmatrix},
		\end{align*}
		we see that  $\text{rank}(\bm{\alpha}_{\Lambda}, \bm{\beta}_{\Lambda},\bm{e}_{\Lambda}, \tilde{\bm{e}})=\text{rank}(\bm{\alpha}, \bm{\beta},\bm{e})+1=\text{rand}(\bm{\alpha}, \bm{\beta})+2$, and $\text{rank}(\bm{\alpha}_{\Lambda}, \bm{\beta}_{\Lambda}, \tilde{\bm{e}})=\text{rank}(\bm{\alpha}, \bm{\beta})+1$, which implies that 
		\begin{equation}\label{e52}
			\text{rank}(\bm{\alpha}_{\Lambda}, \bm{\beta}_{\Lambda}, \tilde{\bm{e}},\bm{\gamma}_{\Lambda})= \text{rank}(\bm{\alpha}_{\Lambda}, \bm{\beta}_{\Lambda}, \tilde{\bm{e}})+1,
		\end{equation}
		noticing that $\text{rank}(\bm{\alpha}_{\Lambda}, \bm{\beta}_{\Lambda}, \tilde{\bm{e}},\bm{\gamma}_{\Lambda})=\text{rank}(\bm{\alpha}_{\Lambda}, \bm{\beta}_{\Lambda}, \bm{e}_{\Lambda}, \tilde{\bm{e}})$.
		
		Write $RA(\bm{q}_{\Lambda})=\langle \bm{q}_{\Lambda}, \bm{\alpha}_{\Lambda} \rangle$ and $SB(\bm{q}_{\Lambda})=\langle \bm{q}_{\Lambda}, \bm{\beta}_{\Lambda} \rangle$, then we have 
		$$
		RA(\bm{q}_{\Lambda})=\big(1+R_1(\bm{q})\big)RA(\bm{q}), \quad 
		SB(\bm{q}_{\Lambda})=\big(1+R_1(\bm{q})\big)SB(\bm{q}).
		$$
		Similarly, write 
		$$
		\begin{aligned}
			QQ(\bm{q}_{\Lambda})&=\sum_{l=1}^v \sum_{m=1}^d q_l q_m \log (q_l q_m)+\sum_{l=v+1}^d q_l\log q_l,
			\\
			RR(\bm{q}_{\Lambda})&=\sum_{i\in \mathcal{J}_1} R_1(\bm{q}) R_i(\bm{q}) \log \big(R_1(\bm{q}) R_i(\bm{q})\big) +\sum_{i\in \mathcal{J}_1:i\neq 1} R_i(\bm{q}) \log  R_i(\bm{q}).
		\end{aligned}
		$$
		We also have $QQ(\bm{q}_{\Lambda})=\big(1+R_1(\bm{q}) \big) QQ(\bm{q})$ and $RR(\bm{q}_{\Lambda})=\big(1+R_1(\bm{q}) \big) RR(\bm{q})$. Thus 
		\begin{equation}\label{e107}
			\dimh K=\frac{RR(\bm{q}_{\Lambda})}{RA(\bm{q}_{\Lambda})}+\frac{QQ(\bm{q}_{\Lambda})-RR(\bm{q}_{\Lambda})}{SB(\bm{q}_{\Lambda})}.
		\end{equation}

		Define
		$$ 
		\Sigma_{\Lambda,1}=\left\{\varepsilon\in {\Lambda}^\mathbb{N}:Q_{\Lambda}(\varepsilon,n) \text{ is of $1$-type for infinitely many } n\in \mathbb{N}  \right\},
		$$
		where $Q_{\Lambda}(\varepsilon,n)$'s are $1$-type approximate squares according to the IFS $\{\psi_\varepsilon:\varepsilon\in \Lambda\}$ (Recall Remark \ref{re1} these squares are well-defined, although the new IFS is not a \bara~IFS). Let $\iota$ be the identity mapping from $\Lambda^{\mathbb{N}}$ to $\Sigma^{\mathbb{N}}$. Observe that for a $1$-type approximate square $Q_{\Lambda}(\varepsilon,n)$, there is $k\in \mathbb{N}$ such that 
		\begin{equation}\label{e111}
			Q\big(\iota(\varepsilon),k\big) \subseteq Q_{\Lambda}(\varepsilon,n), \text{ and } Q\big(\iota(\varepsilon),k\big) \text{ has the same width}.
		\end{equation}  
		Since $Q_{\Lambda}(\varepsilon,n)$ is of $1$-type, by Lemma \ref{le21}, the width of rectangle $\psi_{\varepsilon|_n}([0,1]^2)$ is larger than its height. Noticing that $\varepsilon|_n=\iota(\varepsilon)|_k$, again using Lemma \ref{le21}, we see that $Q\big(\iota(\varepsilon),k\big)$ is also of $1$-type. Thus $\iota(\Sigma_{\Lambda,1})\subseteq \Sigma_1$. 
		
		Let $\Pi_{\Lambda}=\Pi\circ \iota$ be the projection from $\Lambda^{\mathbb{N}}$ to $K$. Using the same argument in Case $1$, noticing that  \eqref{e52}, \eqref{e107} and $RA(\bm{q}_{\Lambda})>SB(\bm{q}_{\Lambda})$, there is a probability measure $\nu_{\Lambda,\delta}$ on $\Lambda^{\mathbb{N}}$ such that $\nu_{\Lambda,\delta}(\Sigma_{\Lambda,1})=1$ and for $\nu_{\Lambda,\delta}$ a.e. $\varepsilon\in \Sigma_{\Lambda,1}$, 
		$$
		\lim_{n\to +\infty}\log \mu_{\Lambda,\delta} \big( Q_\Lambda(\varepsilon,n) \big) -\log \varphi_c(L_{\Lambda,\varepsilon|_n})=-\infty,
		$$
		where $\mu_{\Lambda,\delta}=\nu_{\Lambda,\delta} \circ \Pi_{\Lambda}^{-1}$, and $L_{\Lambda,\varepsilon|_n}$ is the maximum of the width and height of rectangle $\psi_{\varepsilon|_n}([0,1]^2)$.
		Define 
		$$
		\Sigma_{\Lambda,1}'=\{\varepsilon\in \Sigma_{\Lambda,1}:\lim_{n\to\infty} \frac{\varphi_c(L_{\Lambda,\varepsilon|_n})}{\mu_{\Lambda,\delta}\big(Q_{\Lambda}(\varepsilon,n)\big)}=+\infty\}.
		$$
		Then, $\nu_{\Lambda,\delta}(\Sigma_{\Lambda,1}')=1$. Write $\mu_{\delta}:=\mu_{\Lambda,\delta},$ and by \eqref{e111}, we have
		$$
		\iota(\Sigma_{\Lambda,1}') \subseteq \{\varepsilon\in \Sigma_{1}:\lim_{k\to\infty} \frac{\varphi_c(L_{\varepsilon|_k})}{\mu_{\delta}\big(Q(\varepsilon,k)\big)}=+\infty \}:=\Sigma_1'.
		$$
		Using Proposition \ref{prop4}, we have $\mathcal{H}^{\varphi_c}(K_1)\geq \mathcal{H}^{\varphi_c}\big(\Pi(\Sigma_1') \big)=+\infty$.
		\vspace{0.2cm}
		
		\textit{\textbf{Case 3.}} \textit{When there are $a,b$ in $\mathbb{R}$ such that $\bm{\gamma}=a \bm{\alpha} +b \bm{\beta}$.}
		\vspace{0.2cm}

		Since $\langle \bm{q}, \bm{\gamma} \rangle=0$ and $\bm{\gamma}$ is not zero, $\text{rank}(\bm{\alpha},\bm{\beta})=2$ (otherwise $\langle \bm{q}, \bm{\gamma} \rangle =a' RA(\bm{q})\neq 0$ for some $a'\neq 0$), and $a\rho_1 +b=0$. This also gives $a,b\neq 0$.
		
		It suffices to assume that 
		\begin{equation}\label{e56}
			\text{rank}(\bm{\alpha},\bm{e},\bm{\beta})=\text{rank}(\bm{\alpha},\bm{e})+1.
		\end{equation} 
		In fact, otherwise we will have 
		$$
		2\geq \text{rank}(\bm{\alpha},\bm{e})=\text{rank}(\bm{\alpha},\bm{e},\bm{\beta})\geq \text{rank}(\bm{\alpha},\bm{\beta})=2.
		$$ 
		Then instead to consider the same  IFS $\{\psi_{\varepsilon}:\varepsilon\in \Lambda\}$ as shown in Case $2$, taking elementary column transformations, we have 
		$$
		\begin{aligned}
			\text{rank}(\bm{\alpha}_{\Lambda},\bm{e}_{\Lambda}, \tilde{\bm{e}})=\text{rank}(\bm{\alpha},\bm{e})+1=3,
			\quad
			\text{rank}(\bm{\alpha}_{\Lambda}, \tilde{\bm{e}})=\text{rank}(\bm{\alpha})+1=2.
		\end{aligned}
		$$
		Noticing that there are $a',e'$ in $\mathbb{R}$ such that $\bm{\beta}=a' \bm{\alpha}+ e'\bm{e}.$ Thus $\bm{\beta}_{\Lambda}=a'\bm{\alpha}_{\Lambda}+ e'\bm{e}_{\Lambda}$. Since $\text{rank}(\bm{\alpha},\bm{\beta})=2$, $e'\neq 0$. Then  
		$$
		\text{rank}(\bm{\alpha}_{\Lambda}, \tilde{\bm{e}} ,\bm{\beta}_{\Lambda})=
		\text{rank}(\bm{\alpha}_{\Lambda},\bm{e}_{\Lambda}, \tilde{\bm{e}} ) =\text{rank} (\bm{\alpha}_{\Lambda}, \tilde{\bm{e}})+1,
		$$
		a formula similar to \eqref{e56}.  
		\vspace{0.2cm}
		
		Choose a vector $\bm{u}=(u_1,\cdots, u_d)$ in $\mathbb{R}^d$ satisfying 
		\begin{equation}\label{e53}
			\langle \bm{u},\bm{\alpha} \rangle =0,\quad  \langle \bm{u},\bm{e} \rangle =0,\quad \langle \bm{u}, \bm{\beta} \rangle \cdot b > 0.
		\end{equation}
		Write $\Delta:=\langle \bm{u}, \bm{\beta} \rangle$. For small $\delta>0$, define probability measures $\nu_\delta, \mu_{\delta}$ in the same way as described in Case 1 by replacing the new $\bm{u}$. Let $\varepsilon \in \Sigma^{\mathbb{N}}$ satisfy \eqref{e47} with respect to $\nu_{\delta}$. 
		Thus, by \eqref{e53}, equation \eqref{e55} becomes
		\begin{equation}\label{e59}
			\begin{aligned}
				RA(\bm{q}^{(l)})&=\langle \bm{q}^{(l)}, \bm{\alpha} \rangle=\langle \bm{q}+\frac{\delta}{\log l}\bm{u}, \bm{\alpha}\rangle =RA(\bm{q}),
				\\
				SB(\bm{q}^{(l)})&=\langle \bm{q}^{(l)}, \bm{\beta} \rangle=\langle \bm{q}+\frac{\delta}{\log l}\bm{u}, \bm{\beta}\rangle =SB(\bm{q})+ \frac{\delta}{\log l} \Delta.
			\end{aligned}
		\end{equation}
		For sufficiently large $n$, write $k_1(n)=k_1(\varepsilon, L_{\varepsilon|_n})$ and $k_2(n)=k_2(\varepsilon, L_{\varepsilon|_n})$ for short. Note that either $k_2(n)\leq k_1(n)=n$, or $k_1(n)\leq k_2(n)=n$. Combining Lemma \ref{le11}, \eqref{e47} and  \eqref{e59}, we have 
		\begin{equation}\label{e60}
			\Big|k_1(n) RA(\bm{q})-k_2(n) SB(\bm{q}) -\sum_{l=2}^{k_2(n)} \frac{\delta}{\log l} \Delta \Big|=O \big( \sqrt{n\log \log n} \big).
		\end{equation}
		Then $$
		k_1(n)\geq \frac{k_2(n)}{\rho_1}-k_2(n)\frac{\delta| \Delta|}{|RA(\bm{q})|\log 2} + O \big( \sqrt{n\log \log n} \big).
		$$
		Since $RA(\bm{q})>SB(\bm{q})$, $0<\rho_1<1$, we see that $n=k_1(n)\geq k_2(n)$ holds for all large $n$ since $\delta$ is small. So $\nu_{\delta}(\Sigma_1)=1$.

		Still fix the above $\varepsilon$. Combining \eqref{e58} and \eqref{e60}, we have the fact that  
		\begin{equation}\label{e87}
			\begin{aligned}
				&\Big|(k_2(n) -n \rho_1)\big( QQ(\bm{q})-RR(\bm{q}) \big) +\sum_{l=2}^{k_2(n)}\frac{\delta}{\log l}\lambda_1\Delta \Big|
				\\
				=&\Big|\big(k_2(n) -k_1(n) \rho_1\big)\big( QQ(\bm{q})-RR(\bm{q}) \big)+ \sum_{l=2}^{k_2(n)}\frac{\delta}{\log l}\Delta \frac{\big( QQ(\bm{q})-RR(\bm{q}) \big)}{SB(\bm{q})} \Big|= O \big( \sqrt{n\log \log n} \big).
			\end{aligned}
		\end{equation}
		Note that by \eqref{e49} and \eqref{e53}, \eqref{e50} becomes
		\begin{equation}
			\begin{aligned}\nonumber
				QQ(\bm{q}^{(l)}) &= QQ(\bm{q}) + \frac{\delta}{\log l} \big(\lambda_1+b(\rho_1-1)\big)\Delta +O\left(\frac{\delta^2}{(\log l)^2}\right),
				\\
				RR(\bm{q}^{(l)})&= RR(\bm{q}) +\frac{\delta}{\log l} b\rho_1 \Delta +O\left(\frac{\delta^2}{(\log l)^2}\right).
			\end{aligned}
		\end{equation}
		Combining this with \eqref{e47}, \eqref{e98} and \eqref{e87}, we have 
		\begin{equation}\nonumber
			\begin{aligned}
				&\log \mu_{\delta}\big( Q(\varepsilon,n) \big)-nRR(\bm{q}) -n \rho_1\big(QQ(\bm{q})-RR(\bm{q}) \big)
				\\  
				=&\sum_{l=1}^{k_2(n)} QQ(\bm{q}^{(l)}) +\sum_{l=k_2(n)+1}^n RR(\bm{q}^{(l)})-nRR(\bm{q}) -n \rho_1\big(QQ(\bm{q})-RR(\bm{q}) \big)+O \big( \sqrt{n\log \log n} \big)
				\\
				=&\sum_{l=2}^{k_2(n)} \frac{\delta}{\log l}b(\rho_1-1) \Delta+\sum_{l=k_2(n)+1}^n \frac{\delta}{\log l}b\rho_1 \Delta +O\left( \frac{\delta^2 n}{(\log n)^2}\right).
			\end{aligned}
		\end{equation}
		Therefore by \eqref{e57} and \eqref{e59},
		\begin{equation}\label{e88}
			\begin{aligned}
				&\log \mu_{\delta}\big( Q(\varepsilon,n) \big) -\log \varphi(L_{\varepsilon|_n})
				\\
				\leq  &\sum_{l=2}^{k_2(n)} \frac{\delta}{\log l}b(\rho_1-1) \Delta+\sum_{l=k_2(n)+1}^n \frac{\delta}{\log l}b\rho_1 \Delta +c|RA(\bm{q})|\frac{ n}{(\log n)^2} +O\left( \frac{\delta^2 n}{(\log n)^2}\right).
			\end{aligned}
		\end{equation}
		Note that by \eqref{e60},
		\begin{equation}\label{e99}
			\Big| \rho_1 -\frac{k_2(n)}{n} -\frac{\delta \Delta}{n SB(\bm{q})} \sum_{l=2}^{k_2(n)} \frac{1}{\log l } \Big| = O \left( \frac{1}{n}\sqrt{n\log \log n} \right).
		\end{equation}
		Combining this with $\sum_{l=2}^{n} \frac{1}{\log l}= O\left( \frac{n}{\log n} \right)$ and $k_2(n)\leq n$,
		we have 
		$\frac{k_2(n)}{n}\to \rho_1<1$ as $n$ tends to infinity. So 
		\begin{equation}\label{e89}
			\begin{aligned}
				&\sum_{l=2}^{k_2(n)} \frac{\delta}{\log l}b(\rho_1-1) \Delta+\sum_{l=k_2(n)+1}^n \frac{\delta}{\log l}b\rho_1 \Delta 
				\\
				\leq& b\delta \Delta \left( \rho_1 \int_{2}^{n} \frac{1}{\log t}dt-\int_{2}^{k_2(n)}\frac{1}{\log t}dt \right) +O(1)
				\\
				\leq& b\delta \Delta \left( \rho_1 \int_{2}^{n} \frac{1}{\log t}dt-\int_{2}^{n}\frac{k_2(n)}{ n\log (\frac{k_2(n)}{n}t)}dt \right) +O(1)
				\\
				\leq& b\delta\Delta \rho_1 \log \left(\frac{k_2(n)}{n}\right) \frac{n}{(\log n)^2}+ \Big|b\delta\Delta \left(\rho_1-\frac{k_2(n)}{n}\right)\int_{2}^n \frac{1}{\log \left(\frac{k_2(n) t}{n} \right)}dt \Big| +O(1) .
			\end{aligned}
		\end{equation}
		Again using \eqref{e99}, we have 
		$$
		\begin{aligned}
			&\Big|b\delta\Delta \left(\rho_1-\frac{k_2(n)}{n}\right)\int_{2}^n \frac{1}{\log \left(\frac{k_2(n) t}{n} \right)}dt \Big| 
			\\
			=& \big|b\delta\Delta\big|\cdot \Big| \rho_1-\frac{k_2(n)}{n} \Big| \cdot \Big| \int_{2}^n \frac{1}{\log \left(\frac{k_2(n) t}{n} \right)}dt \Big| 
			= O \left(\frac{\delta^2 n}{(\log n)^2} \right). 
		\end{aligned}
		$$
		Therefore, combining this with \eqref{e88} and \eqref{e89},  we have 
		$$
		\begin{aligned}
			&\log \mu_{\delta}\big( Q(\varepsilon,n) \big)-\log \varphi(L_{\varepsilon|_n})
			\\
			\leq &b\delta\Delta \rho_1 \log \left(\frac{k_2(n)}{n}\right) \frac{n}{(\log n)^2} + c|RA(\bm{q})|\frac{ n}{(\log n)^2} +O\left( \frac{\delta^2 n}{(\log n)^2}\right).
		\end{aligned}
		$$
		Picking $\delta,c $ small enough as in Case 1, using  $b\Delta>0$, $\log \left(\frac{k_2(n)}{n}\right) \to \log \rho_1 <0$, we have 
		$$
		\log \mu_{\delta}\big( Q(\varepsilon,n) \big)-\log \varphi(L_{\varepsilon|_n})\to -\infty, \quad \text{ as }n\to +\infty.
		$$
		This immediately gives $\mathcal{H}^{\varphi_c}(K_1)=+\infty$ as in Case $1$.
		
	\end{proof}

	\begin{proposition}\label{le22}
		Let $K$ be a \bara~carpet. If the condition \ref{H} does not hold, then $$
		\mathcal{H}^{\varphi_c}(K)=+\infty
		$$
		holds for some small $c>0$,
		where $\varphi_c$ is defined in \eqref{e90}.
	\end{proposition}
	\begin{proof}
		Using Proposition \ref{le15}, since the condition \ref{H} does not hold, we see that 
		$$
		\begin{aligned}
			\text{either }&\quad  G_1\geq G_2 \text{ and } K \text{ does not have uniform vertical fibres},
			\\
			\text{or }& \quad  G_2\geq G_1  \text{ and } K \text{ does not have uniform horizontal fibres}.
		\end{aligned}
		$$
		Combining this with Lemmas \ref{le7} and \ref{le18}, we have 
		$$
		\text{either }\quad \mathcal{H}^{\varphi_c}(K)\geq \mathcal{H}^{\varphi_c}(K_1)=+\infty, \quad \text{ or }\quad \mathcal{H}^{\varphi_c}(K)\geq \mathcal{H}^{\varphi_c}(K_2)=+\infty.
		$$
		The proposition follows.
	\end{proof}
	\begin{proof}[Proof of ``$\Longrightarrow$'' in Theorem \ref{th1}-\eqref{e41}]
		This follows by combining the fact that  $\varphi_c(t)\leq t^{\dimh K}$ for all $t$, and Proposition \ref{le22}.
	\end{proof}

	\section{Proof of Theorem \ref{th1}-\eqref{e42}  }\label{sec6}
	In this section, we prove that 
	$
	\ref{B} \iff \dimh K=\dimb K.
	$

	First of all, we recall a result from Feng and Wang \cite{FW05}, which reformulates the box dimension of a \bara~carpet as the maximum of a function $f$ defined on $\mathcal{S}$.

	\begin{proposition}(\cite[Theorem 1]{FW05}) \label{prop1}
		For a \bara~carpet $K$,  $\dimb K=\max_{\bm{q}\in \mathcal{S}} f(\bm{q})$, where 
		$$
		f(\bm{q})=\left \{\begin{aligned}
			&\frac{\sum_{l=1}^d q_l \log q_l -t_1\big(\sum_{i=1}^r R_i(\bm{q})\log a_i- \sum_{j=1}^s S_j(\bm{q})\log b_j \big)}{\sum_{j=1}^s S_j(\bm{q})\log b_j},\quad & & \text{ if } \bm{q}\in \mathcal{S}_1, \\
			&\frac{\sum_{l=1}^d q_l \log q_l -t_2 \big(\sum_{j=1}^s S_j(\bm{q})\log b_j-\sum_{i=1}^r R_i(\bm{q})\log a_i\big )}{\sum_{i=1}^r R_i(\bm{q})\log a_i}, & &\text{ if } \bm{q} \in \mathcal{S}_2 \setminus \mathcal{S}_1,
		\end{aligned} \right.
		$$
		and $t_1 =\dimb \pi_1(K)$, $t_2=\dimb \pi_2(K)$.
	\end{proposition}
	
	Proposition \ref{prop1} is a simplified version of \cite[Theorem 1]{FW05}, which characterizes the $L^q$-spectra of self-affine measures on box-like self-affine carpets.
	\vspace{0.2cm}
	
	Similar to the Hausdorff dimension case \eqref{e6}, we  define two functions  $f_1,f_2:\mathcal{S}\to \mathbb{R}$ by 
	$$
	\begin{aligned}
		f_1(\bm{q}):&=\frac{QQ(\bm{q}) -t_1\big (RA(\bm{q})- SB(\bm{q}) \big )}{SB(\bm{q})},\\
		f_2(\bm{q}):&=\frac{QQ(\bm{q}) -t_2 \big (SB(\bm{q})-RA(\bm{q})\big )}{RA(\bm{q})}.
	\end{aligned}
	$$
	Then 
	$$
	f(\bm{q})=\left \{\begin{aligned}
		&f_1(\bm{q}), \quad & & \text{ if } \bm{q}\in \mathcal{S}_1, \\
		&f_2(\bm{q}), & &\text{ if } \bm{q} \in \mathcal{S}_2 \setminus \mathcal{S}_1.
	\end{aligned} \right.
	$$
	
	So $\dimb K=\max_{\bm{q}\in \mathcal{S}} f(\bm{q})=\max \{ D_1,D_2\}$, recalling Proposition \ref{prop2},
	where $D_1,D_2$ are the unique real numbers such that 
	\begin{equation}\nonumber
		\sum_{l=1}^d a_{i_l}^{t_1} b_{j_l}^{D_1-t_1}=1, \quad\quad \sum_{l=1}^d b_{j_l}^{t_2} a_{i_l}^{D_2-t_2}=1.
	\end{equation}

	\begin{lemma}\label{le3} 
		Let $\bm{q}^{(1)}=(a_{i_l}^{t_1} b_{j_l}^{D_1-t_1})_{l=1}^d$, $\bm{q}^{(2)}=(b_{j_l}^{t_2} a_{i_l}^{D_2-t_2})_{l=1}^d$ be two probability vectors.
		\begin{enumerate}[(a)]
			\item \label{e30}  $\bm{q}^{(1)}$ and $\bm{q}^{(2)}$ are maximum points of functions $f_1$ and $f_2$ in $\mathcal{S}$, respectively, and  
			$$\begin{aligned}
				&\max_{\bm{q}\in \mathcal{S}}f_1(\bm{q})=f_1(\bm{q}^{(1)})=D_1,\\    &\max_{\bm{q}\in \mathcal{S}}f_2(\bm{q})=f_2(\bm{q}^{(2)})=D_2,
			\end{aligned}$$
			\item \label{e31}
			when $\mathcal{J}_1\times \mathcal{J}_2\neq \mathcal{J}$, if $D_1\geq D_2$, we have $\bm{q}^{(1)} \in \emph{int}(\mathcal{S}_1)$; if $D_2\geq	 D_1$, we have $\bm{q}^{(2)}\in \emph{int}(\mathcal{S}_2)$.
		\end{enumerate}
	\end{lemma}
	\begin{proof}
		Part \eqref{e30} is derived from a similar calculation used in the proofs of Lemmas \ref{le5} and \ref{le6}, employing the Lagrange multipliers method with  objective functions $f_1,f_2$ under the constraint $\bm{q}\in \mathcal{S}$.
		
		For part \eqref{e31}, recall the function $y(x)$ defined in Lemma \ref{le16}, we have 
		$$
		D_1-t_1=y(t_1),\quad t_2=y(D_2-t_2).
		$$
		Since $\mathcal{J}_1\times \mathcal{J}_2\neq \mathcal{J}$, we have $\dimb K< t_1+t_2$ by the product formula (see \cite{F90}), which gives that  $D_2-t_2<t_1$. When $D_1\geq D_2$, since $y''(x)>0$, we have  $y'(t_1)>-1$ (otherwise, noticing that  the function $x+y(x)$ on $[D_2-t_2, t_1]$ is strictly decreasing, we would have $D_2>D_1$, a contradiction), then combining this with  \eqref{e8}, we have  $RA(\bm{q}^{(1)}) >SB(\bm{q}^{(1)})$, so  $\bm{q}^{(1)}\in \text{int}(\mathcal{S}_1)$. The case $D_2\geq D_1$ follows similarly.
		
	\end{proof}

	\begin{lemma}\label{le4}The following statements are true:
		\begin{enumerate}[(a)]
			\item \label{e4}
			for all $\bm{q}\in \emph{int}(\mathcal{S}_1)$, $g_1(\bm{q})\leq f_1(\bm{q})$, and equality holds if and only if $\frac{RR(\bm{q})}{RA(\bm{q})}=t_1$,
			\item \label{e9}
			for all $\bm{q}\in \emph{int}(\mathcal{S}_2)$, $g_2(\bm{q})\leq f_2(\bm{q})$, and equality holds if and only if $\frac{SS(\bm{q})}{SB(\bm{q})}=t_2$.
		\end{enumerate}
	\end{lemma}
	\begin{proof}
		It suffices to prove \eqref{e4}. Note that  
		$$
		\begin{aligned}
			f_1(\bm{q})-g_1(\bm{q})&=\frac{QQ(\bm{q})-t_1\big (RA(\bm{q})-SB(\bm{q})\big )}{SB(\bm{q})}-\frac{QQ(\bm{q})}{SB(\bm{q})}-RR(\bm{q})\left(\frac{1}{RA(\bm{q})}-\frac{1}{SB(\bm{q})}\right)
			\\
			&=\left( \frac{1}{SB(\bm{q})}-\frac{1}{RA(\bm{q})}\right)\left( \frac{RR(\bm{q})}{RA(\bm{q})}-t_1 \right) RA(\bm{q}).
		\end{aligned}
		$$
		Recall that $t_1=\dimh \pi_1(K)$, $\pi_1(K)$ is a self-similar set satisfying OSC, so $t_1$ is the maximum of dimension of self-similar measures, i.e.
		$$
		t_1=\max_{\bm{q}\in \mathcal{S}} \frac{RR(\bm{q})}{RA(\bm{q})}.
		$$
		If $\bm{q}\in \text{int}(\mathcal{S}_1)$, then $0>RA(\bm{q})>SB(\bm{q})$. So  $f_1(\bm{q})-g_1(\bm{q})\geq 0$, and equality holds if and only if $\frac{RR(\bm{q})}{RA(\bm{q})}=t_1$.
	\end{proof}
	
	\begin{lemma}\label{le19}
		For a \bara~carpet $K$, when $\mathcal{J}_1 \times \mathcal{J}_2 \neq \mathcal{J}$, we have
		\begin{enumerate}[(a)]
			\item \label{e69} if $G_1\geq G_2$ and $G_1=D_1$, then $K$ has uniform vertical fibres,
			\item \label{e70} if $G_2\geq G_1$ and $G_2=D_2$, then $K$ has uniform horizontal fibres.
		\end{enumerate}
	\end{lemma}
	\begin{proof}
		Clearly, we only need to prove \eqref{e69}.
		
		By Lemma \ref{le5}, for any $\bm{q}^{(1)}$ with $g_1(\bm{q}^{(1)})=G_1$, $\bm{q}^{(1)}$ has the following form:
		$$
		\bm{q}^{(1)}=\left( a_{i_l}^{\theta_1} b_{j_l}^{\lambda_1} \big( \sum_{m\in I_{i_l}} b_{j_m}^{\lambda_1} \big)^{\rho_1-1} \right)_{l=1}^d
		$$
		where
		$$
		\theta_1=\frac{RR(\bm{q}^{(1)})}{RA(\bm{q}^{(1)})}, \quad 
		\lambda_1=\frac{QQ(\bm{q}^{(1)})-RR(\bm{q}^{(1)})}{SB(\bm{q}^{(1)})}, \quad
		\rho_1=\frac{RA(\bm{q}^{(1)})}{SB(\bm{q}^{(1)})}.
		$$
		By Proposition \ref{le8}, $\bm{q}^{(1)}\in \text{int}(\mathcal{S}_1)$. Thus $0<\rho_1<1$. Using Lemmas \ref{le3} and  \ref{le4}, noticing that $G_1=g_1(\bm{q}^{(1)})\leq f_1(\bm{q}^{(1)})\leq D_1=G_1$, we have $\theta_1=t_1$. Noticing that $\theta_1+\lambda_1=g_1(\bm{q}^{(1)})=G_1=D_1$, we have $\lambda_1=D_1-t_1$. Define a function $\phi:[0,1]\to \mathbb{R}$ by 
		$$
		\phi(\rho)=\sum_{l=1}^d a_{i_l}^{t_1} b_{j_l}^{D_1-t_1} \big( \sum_{m\in I_{i_l}} b_{j_m}^{D_1-t_1}\big)^{\rho-1}.
		$$
		Note that $\phi(0)=\phi(1)=\phi(\rho_1)=1$, $\rho_1\in (0,1)$, and $\phi''(\rho)\geq 0$, where $$
		\phi''(\rho)=\sum_{l=1}^d a_{i_l}^{t_1} b_{j_l}^{D_1-t_1} \big( \sum_{m\in I_{i_l}} b_{j_m}^{D_1-t_1}\big)^{\rho-1}  \big(\log  \sum_{m\in I_{i_l}} b_{j_m}^{D_1-t_1}\big)^2.
		$$
		This gives that  $\phi''\equiv 0$ on $[0,1]$. So for each $l=1,\cdots,d$, $\sum_{m\in I_{i_l}} b_{j_m}^{D_1-t_1}=1$. Therefore, $K$ has uniform vertical fibres.
	\end{proof}
	\begin{proof}[Proof of Theorem \ref{th1}-\eqref{e42}] If $\mathcal{J}_1 \times \mathcal{J}_2 = \mathcal{J}$, there is nothing to prove since $K$ has both uniform vertical and horizontal fibres and $\dimh K=\dimb K=t_1+t_2$. It suffices to assume $\mathcal{J}_1 \times \mathcal{J}_2 \neq \mathcal{J}$.
		\vspace{0.2cm}
		
		``$\Longrightarrow$''. Without loss of generality, assume that $D_1\geq D_2$ and $K$ has uniform vertical fibres. Then there is $t\geq 0$ such that
		$$
		\sum_{l\in I_i} b_{j_l}^t=1,\quad \text{ for all } i\in \mathcal{J}_1.
		$$
		Combining this with \eqref{e1} and \eqref{e10}, we have $t=D_1-t_1$. Let $\bm{q}^{(1)}=(a_{i_l}^{t_1} b_{j_l}^{D_1-t_1})_{l=1}^d$. By Lemma \ref{le3}, $\bm{q}^{(1)} \in \text{int}(\mathcal{S}_1)$. Note that 
		$$
		\frac{RR(\bm{q}^{(1)})}{RA(\bm{q}^{(1)})}=\frac{\sum_{l=1}^d a_{i_l}^{t_1} b_{j_l}^{D_1-t_1}\log \big( a_{i_l}^{t_1} \sum_{m\in I_{i_l}} b_{j_{m}}^{D_1-t_1} \big)}{RA(\bm{q}^{(1)})}=\frac{\sum_{l=1}^d a_{i_l}^{t_1} b_{j_l}^{D_1-t_1}\log  a_{i_l}^{t_1} }{RA(\bm{q}^{(1)})}=t_1.
		$$
		Using Lemmas \ref{le3} and  \ref{le4}, we have  $\max_{\bm{q}\in \mathcal{S}_1} g_1(\bm{q})\geq g_1(\bm{q}^{(1)})= f_1(\bm{q}^{(1)})=D_1$. So $\dimh K=\max_{\bm{q}\in \mathcal{S}} g(\bm{q})\geq D_1=\dimb K$, which gives that $\dimh K=\dimb K$.
		\vspace{0.2cm}
		
		``$\Longleftarrow$''. 
		When $D_1>D_2$, we claim that $G_1\geq G_2$. Noticing that if $\dimh K=G_2>G_1,$ by Proposition \ref{le8}, for any $\bm{q}^{(2)}$ with $g_2(\bm{q}^{(2)})=G_2$, $\bm{q}^{(2)}\in \text{int}(\mathcal{S}_2)$. By Lemmas \ref{le3} and \ref{le4}, we have $G_2=g_2(\bm{q}^{(2)})\leq \max_{\bm{q}\in \mathcal{S}_2} f_2(\bm{q})\leq D_2$. Thus $\dimh K\leq D_2<D_1=\dimb K$, a contradiction. Therefore, $D_1=G_1\geq G_2$. Combining this with Lemma \ref{le19}, $K$ has uniform vertical fibres. The case of $D_2>D_1$ follows in a similar way. When $D_1=D_2$, then either $G_1=D_1$ or $G_2=D_2$. Still by Lemma \ref{le19},  $K$ has uniform vertical or horizontal fibres.
		
	\end{proof}

	\section{Proof of ``$\Longrightarrow$'' in Theorem \ref{th1}-\eqref{e40}}\label{sec5}
	In this section, we provide the proof of ``$\Longrightarrow$'' in Theorem \ref{th1}-\eqref{e40}, i.e. 
	$$
	\ref{H} \Longrightarrow 0<\mathcal{H}^{\dimh K}(K)<+\infty.
	$$

	For $\delta>0$, let 
	\begin{equation}\label{e112}
		\mathcal{Q}_\delta=\{Q(\varepsilon,\delta):\varepsilon\in {\Sigma}^\mathbb{N}\},\quad
		\mathcal{U}_\delta=\{U(\varepsilon,\delta):\varepsilon\in {\Sigma}^\mathbb{N}\}.
	\end{equation}
	Write $\mathcal{Q}=\bigcup_{\delta>0} \mathcal{Q}_{\delta}$ and $\mathcal{U}=\bigcup_{\delta>0} \mathcal{U}_{\delta}$. It is easy to see that $\mathcal{Q}_\delta$ is a cover of $K$, $\mathcal{U}_\delta$ is a disjoint cover of ${\Sigma}^\mathbb{N}$, and there is an one-to-one mapping from $\mathcal{Q}$ to $\mathcal{U}$, denoted by $\Theta$. Note that for $Q\neq Q'\in \mathcal{Q}_{\delta}$, $\Theta(Q)\cap \Theta(Q')=\emptyset.$ Recall that in \eqref{e65},
	$$
	\begin{aligned}
		\Sigma_1&=\left\{\varepsilon\in {\Sigma}^\mathbb{N}:Q(\varepsilon,n) \text{ is of $1$-type for infinitely many } n\in \mathbb{N}  \right\},
		\\
		\Sigma_2&=\left\{\varepsilon\in {\Sigma}^\mathbb{N}:Q(\varepsilon,n) \text{ is of $2$-type for infinitely many } n\in \mathbb{N} \right\},
	\end{aligned}
	$$
	and $K_1=\Pi(\Sigma_1),K_2=\Pi(\Sigma_2)$.

	\begin{lemma}\label{le13}
		For a \bara~carpet $K$, the following statements are true:
		\begin{enumerate}[(a)]
			\item \label{e32}
			if $K$ has uniform vertical fibres, then $\mathcal{H}^{D_1}(K)>0$ and $\mathcal{H}^{D_1}(K_1)<+\infty$,
			\item \label{e33}
			if $K$ has uniform horizontal fibres, then $\mathcal{H}^{D_2}(K)>0$ and $\mathcal{H}^{D_2}(K_2)<+\infty$.
		\end{enumerate}
	\end{lemma}
	\begin{proof}
		We only need to prove \eqref{e32}.
		\vspace{0.2cm}

		``$\mathcal{H}^{D_1}(K)>0$''. Let $\bm{q}=\big(a_{i_l}^{t_1} b_{j_l}^{D_1-t_1}\big)_{l=1}^d\in \mathcal{S}$. Let $\nu$, $\mu$ be the Bernoulli and self-affine measures associated with $\bm{q}$, respectively. Since $K$ has uniform vertical fibres, there is  $t\geq 0$ such that
		$$
		\sum_{l\in I_i} b_{j_l}^t=1,\quad \text{ for all } i\in \mathcal{J}_1.
		$$
		Combining this with  \eqref{e1} and \eqref{e10}, we have $t=D_1-t_1$. Thus, for $i\in \mathcal{J}_1$ and $j\in \mathcal{J}_2$, we have  
		\begin{equation}\label{e66}
			R_i(\bm{q})=\sum_{l\in I_i} a_{i}^{t_1} b_{j_l}^{D_1-t_1}=a_i^{t_1}, \quad
			S_j(\bm{q})=\sum_{l\in J_j} a_{i_l}^{t_1} b_{j}^{D_1-t_1}\leq b_j^{D_1-t_1}.
		\end{equation}
		For $\varepsilon\in {\Sigma}^\mathbb{N}$ and $\delta>0$, write $k_1=k_1(\varepsilon,\delta)$ and $k_2=k_2(\varepsilon,\delta)$ for short.
		Using Lemma \ref{le12} and \eqref{e66}, if $Q(\varepsilon,\delta)$ is of $1$-type, we have
		\begin{equation}\label{e34}
			\mu\big(Q(\varepsilon,\delta)\big)=q_{l_1}\cdots q_{l_{k_2}} a_{i_{l_{k_2+1}}}^{t_1}\cdots a_{i_{l_{k_1}}}^{t_1}=A_{\varepsilon|_{k_1}}^{t_1} B_{\varepsilon|_{k_2}}^{D_1-t_1},
		\end{equation}
		and if $Q(\varepsilon,\delta)$ is of $2$-type, we have 
		$$
		\mu\big(Q(\varepsilon,\delta)\big)\leq q_{l_1}\cdots q_{l_{k_1}} b_{j_{l_{k_1+1}}}^{D_1-t_1}\cdots b_{j_{l_{k_2}}}^{D_1-t_1}=A_{\varepsilon|_{k_1}}^{t_1} B_{\varepsilon|_{k_2}}^{D_1-t_1}.
		$$
		Combining this with Lemma \ref{le11}, we have
		$$
		\mu\big(Q(\varepsilon,\delta)\big)\leq C^{D_1} \cdot \delta^{D_1}
		$$
		where the constant $C$ is the same in Lemma \ref{le11}. Thus, by the mass distribution principle, $\mathcal{H}^{D_1}(K)>0$.
		\vspace{0.2cm}
		
		``$\mathcal{H}^{D_1}(K_1)<+\infty$''.
		Note that for $\varepsilon\in \Sigma_1$, there are infinitely many $n\in \mathbb{N}$ such that 
		$Q(\varepsilon,n)$ is of $1$-type. In the following, for $\delta>0$, we will inductively construct a cover of $K_1$ using approximate square with ``radius'' less than $\delta$. Let 
		$$
		\mathcal{U}^{(1)}=\{U({\varepsilon,1}):\varepsilon \in \Sigma_1, U(\varepsilon,1) \text{ is of }1\text{-type}, L_{\varepsilon|_1}\leq \delta\}.
		$$
		Suppose that $\mathcal{U}^{(k)}$ is defined, let 
		$$
		\begin{aligned}
			\mathcal{U}^{(k+1)}=\Big\{U(\varepsilon,k+1):\varepsilon\in \Sigma_1, &U(\varepsilon,k+1) \text{ is of }1\text{-type, } L_{\varepsilon|_{k+1}}\leq \delta, 
			\\
			&U(\varepsilon,k+1)\cap U'=\emptyset \text{ for all }U'\in \bigcup_{n= 1,\cdots,k}\mathcal{U}^{(n)}\Big \}.
		\end{aligned}
		$$
		Write $\tilde{\mathcal{U}}=\bigcup_{n\geq  1} \mathcal{U}^{(n)}$ and $\tilde{\mathcal{Q}}=\Theta^{-1}(\tilde{\mathcal{U}})$. For each $\varepsilon\in \Sigma_1$, by choosing 
		$$
		k=\min \Big\{l\in \mathbb{N}: U(\varepsilon,l) \text{ is of $1$-type and } L_{\varepsilon|_l}\leq \delta \Big\},
		$$ 
		it is directly to see that $U(\varepsilon,k)\in \mathcal{U}^{(k)}$.
		Therefore, by the construction, $\tilde{\mathcal{U}}$ is a disjoint cover of $\Sigma_1$ and $\tilde{\mathcal{Q}}$ is a  cover of $K_1$ with ``radius'' less than $\delta$.

		\vspace{0.2cm}
		
		For $Q(\varepsilon,n)\in \tilde{\mathcal{Q}}$, since $Q(\varepsilon,n)$ is of $1$-type, by Lemma \ref{le21}, we have $L_{\varepsilon|_n}=A_{\varepsilon|_n},$ and so $k_1(\varepsilon,L_{\varepsilon|_n})=n$, $k_2:=k_2(\varepsilon,L_{\varepsilon|_n})\leq n$. Then  by Lemmas \ref{le11}, \ref{le12} and \eqref{e34}, we have 
		$$
		\begin{aligned}
			\big|Q(\varepsilon,n)\big|^{D_1}&\leq (\sqrt{2}C A_{\varepsilon|_n})^{D_1}\leq (\sqrt{2}C)^{D_1} A_{\varepsilon|_{n}}^{t_1} B_{\varepsilon|_{k_2}}^{D_1-t_1}
			\\
			&=(\sqrt{2}C)^{D_1}\mu\big(Q(\varepsilon,n)\big)=(\sqrt{2}C)^{D_1}\nu\big(\Theta(Q(\varepsilon,n))\big).
		\end{aligned}
		$$
		Therefore,
		$$
		\mathcal{H}_{\delta}^{D_1}(K_1)\leq \sum_{Q\in \tilde{\mathcal{Q}}} |Q|^{D_1} \leq \sum_{Q\in \tilde{\mathcal{Q}}} (\sqrt{2}C)^{D_1}\nu\big(\Theta(Q)\big)\leq (\sqrt{2}C)^{D_1},
		$$
		which by letting $\delta \to 0$ gives that $\mathcal{H}^{D_1}(K_1)<+\infty$.
		
	\end{proof}
	
	For $\varepsilon=\varepsilon_1\cdots\varepsilon_n\in \Sigma^*$, and $l=1,\cdots, d$, define 
	$$
	\varrho_{\varepsilon}^{(l)}= \frac{\#\{m\leq n:\varepsilon_m=l\}}{n},\quad\quad \bm{\varrho}_{\varepsilon}= \left(\varrho_{\varepsilon}^{(1)},\cdots, \varrho_{\varepsilon}^{(d)} \right).
	$$
	Note that $\bm{\varrho}_{\varepsilon} \in \mathcal{S}$.
	Let $\mathcal{P}\subseteq \mathcal{S}$ be a closed subset, define 
	$$
	\Sigma({\mathcal{P}})=\left \{\varepsilon\in \Sigma^{\mathbb{N}}: \text{all accumulation points of sequence } \{\bm{\varrho}_{\varepsilon|_n}\}_{n\geq 1} \text{ belong to } \mathcal{P} \right\},
	$$
	and write $K(\mathcal{P})=\Pi\big(\Sigma(\mathcal{P})\big)$. We have a following observation of the estimation  ``$\dimh K\leq \max_{\bm{q}\in \mathcal{S}} g(\bm{q})$'' in \cite[pp. 232-235]{B07}, through a local dimension estimation technique.

	\begin{lemma}\label{le14}
		Let $\mathcal{P}\subseteq \mathcal{S}$ be a closed subset. We have
		$$
		\dimh K(\mathcal{P})\leq \max_{\bm{q}\in \mathcal{P}} g(\bm{q}).
		$$
	\end{lemma}
	
	\begin{proof}[Proof of ``$\Longrightarrow$'' in Theorem \ref{th1}-\eqref{e40}]
		If $\mathcal{J}_1\times \mathcal{J}_2 =\mathcal{J},$ $K$ has both uniform vertical and horizontal fibres. Using Lemma \ref{le13}, we have 
		\begin{equation} \label{e103} 
			\begin{aligned}
				&\mathcal{H}^{D_1}(K)>0, \quad \mathcal{H}^{D_1}(K_1)<+\infty,
				\\
				&\mathcal{H}^{D_2}(K)>0, \quad \mathcal{H}^{D_2}(K_2)<+\infty.
			\end{aligned}
		\end{equation}
		By the definitions of $D_1$ and $D_2$ in \eqref{e10} and using the product formula, we find that $D_1=D_2=t_1+t_2=\dimh K.$ Combining this with \eqref{e103}, we have 
		$$
		0<\mathcal{H}^{\dimh K}(K)\leq \mathcal{H}^{D_1}(K_1)+\mathcal{H}^{D_2}(K_2)<+\infty.
		$$

		Next, we consider the case $\mathcal{J}_1 \times \mathcal{J}_2 \neq \mathcal{J}$. By Proposition \ref{le15}, conditions \ref{H} and \ref{H'} are equivalent. When $G_1>G_2$ and $K$ has uniform vertical fibres, using Lemma \ref{le13}, we have 
		\begin{equation}\label{e68}
			\mathcal{H}^{D_1}(K)>0, \quad\quad \mathcal{H}^{D_1}(K_1)<+\infty.
		\end{equation}
		Let $\Sigma_1^c=\Sigma^{\mathbb{N}} \setminus \Sigma_1$ and $K_1^c=\Pi(\Sigma_1^c)$. For $\varepsilon\in \Sigma_1^c$, notice that for $n\geq 1 $,
		$$
		A_{\varepsilon|_n}=\prod_{l=1}^{d}a_{i_l}^{n{\varrho}_{\varepsilon|_n}^{(l)}}, \quad \quad
		B_{\varepsilon|_n}=\prod_{l=1}^{d}b_{j_l}^{n{\varrho}_{\varepsilon|_n}^{(l)}},
		$$
		so
		\begin{equation}\label{e67}
			\log A_{\varepsilon|_n}=n \sum_{l=1}^d {\varrho}_{\varepsilon|_n}^{(l)} \log a_{i_l}, \quad
			\log B_{\varepsilon|_n}=n \sum_{l=1}^d {\varrho}_{\varepsilon|_n}^{(l)} \log b_{j_l}.
		\end{equation}
		Since for all sufficiently large $n$, $Q(\varepsilon,n)=Q(\varepsilon,L_{\varepsilon|_n})$ is of $2$-type, by Lemma \ref{le21}, we know that $A_{\varepsilon|_n}\leq B_{\varepsilon|_n}$. Combining this with \eqref{e67}, we see that 
		$$
		\sum_{l=1}^d \varrho_{\varepsilon|_n}^{(l)} \log a_{i_l} \leq \sum_{l=1}^d \varrho_{\varepsilon|_n}^{(l)} \log b_{j_l}.
		$$
		Hence for any accumulation point $\bm{\varrho}$ of the sequence  $\{\bm{\varrho}_{\varepsilon|_n}\}_{n\geq 1}$, 
		$$
		RA(\bm{\varrho}) \leq SB(\bm{\varrho}),
		$$ 
		which gives that $\bm{\varrho}\in \mathcal{S}_2$. Therefore, $\Sigma_1^c\subseteq \Sigma(\mathcal{S}_2)$. Using Lemma \ref{le14} and \eqref{e84}, we see that  
		\begin{equation}\label{e71}
			\dimh K_1^c\leq \dimh K(\mathcal{S}_2)\leq \max_{\bm{q}\in \mathcal{S}_2} g(\bm{q})= \max_{\bm{q}\in \mathcal{S}_2} g_2(\bm{q}) \leq G_2<G_1. 
		\end{equation}
		Note that by Proposition \ref{le8}, for any $\bm{q}\in \mathcal{S}$ with $g_1(\bm{q})=G_1$, $\bm{q}\in \text{int}(\mathcal{S}_1)$. Using Lemmas \ref{le3} and \ref{le4}, we have 
		\begin{equation}\label{e106}
			G_1=\max_{\bm{q}\in \text{int}(\mathcal{S}_1)} g_1(\bm{q})\leq \max_{\bm{q}\in \mathcal{S}_1} f_1(\bm{q}) \leq D_1.
		\end{equation}
		Combining this with \eqref{e68} and \eqref{e71}, we know that $\dimh K=G_1=D_1$ and 
		$$
		0<\mathcal{H}^{\dimh K}(K) \leq \mathcal{H}^{\dimh K}(K_1^c)+\mathcal{H}^{\dimh K }(K_1)=\mathcal{H}^{\dimh K }(K_1)<+\infty.
		$$
		
		When $G_2>G_1$, using a same argument as above, we also have $0<\mathcal{H}^{\dimh K}(K)<+\infty$. 
		
		When $G_1=G_2$, using Lemma \ref{le13}, we see that  \eqref{e103} holds.
		Then $\dimb K=\max\{D_1,D_2\} \leq \dimh K=G_1=G_2$. By Proposition \ref{le8}, Lemmas \ref{le3} and \ref{le4}, we have \eqref{e106} and 
		$$
		G_2=\max_{\bm{q}\in \text{int}(\mathcal{S}_2)} g_2(\bm{q})\leq \max_{\bm{q}\in \mathcal{S}_2} f_2(\bm{q}) \leq D_2,
		$$
		which implies that $\dimh K=D_1=D_2.$ Therefore, by \eqref{e103}, we also have $0<\mathcal{H}^{\dimh K}(K)<+\infty$.
		
	\end{proof}
	
	\begin{remark}\label{re3}
		It follows from the proof of ``$\Longrightarrow$'' in Theorem \ref{th1}-\eqref{e40}, when the condition \ref{H} holds, we have
		$$
		\text{ either } \quad \dimh K=D_1 \quad  \text{ or } \quad \dimh K=D_2.
		$$
	\end{remark}

	\section{Proof of Theorem \ref{th3}}\label{sec7}
	In this section, we will prove Theorem \ref{th3}.  Firstly, we prove that 
	$$
	``\ref{A}  \iff \dimb K=\dima K  \iff  \dimh K =\dima K \text{''}.
	$$

	Recall the definition of $E_1$, $E_2$ in \eqref{e79}, 
	\begin{equation} \nonumber
		\left\{
		\begin{aligned}
			E_1&=t_1+\max_{l\in \{1,\cdots, d\}} S_{1,l},& & E_2=-1, \quad & & \text{ if } K \text{ is of horizontal type},
			\\
			E_1&=-1,\quad& & E_2=t_2+\max_{l\in \{1,\cdots, d\}} S_{2,l}, & & \text{ if } K \text{ is of vertical type},
			\\
			E_1&=t_1+\max_{l\in \{1,\cdots, d\}} S_{1,l},\quad& & E_2=t_2+\max_{l\in \{1,\cdots, d\}} S_{2,l}, & & \text{ if } K \text{ is of mixed type}.
		\end{aligned} \right.
	\end{equation}
	Note that $S_{1,l}$ (resp. $S_{2,l}$) is the Hausdorff dimension of the self-similar set generated by IFS $\{\pi_2\circ\psi_{m}\circ\pi_2^{-1}:m\in I_{i_l}\}$ (resp. $\{\pi_1\circ\psi_{m}\circ\pi_1^{-1}:m\in J_{j_l}\}$), and $K$ has uniform vertical (resp. horizontal) fibres if and only if all $S_{1,l}$ (resp. $S_{2,l}$) are the same. Write
	\begin{equation}\label{e100}
		\tilde{E}_1=t_1+\max_{l\in \{1,\cdots, d\}} S_{1,l},\quad \tilde{E}_2=t_2+\max_{l\in \{1,\cdots, d\}} S_{2,l}. 
	\end{equation}

	\begin{lemma}\label{le9}
		For a \bara~carpet $K$, the following statements are true:
		\begin{enumerate}[(a)]
			\item \label{e13}
			$D_1\leq \tilde{E}_1$, and equality holds if and only if $K$ has uniform vertical fibres,
			\item \label{e14}
			$D_2\leq \tilde{E}_2$, and equality holds if and only if $K$ has uniform horizontal fibres.
		\end{enumerate}
	\end{lemma}
	\begin{proof}
		It suffices to provide the proof for part \eqref{e13}.
		\vspace{0.2cm}
		
		Combining the definitions of $t_1$, $S_{1,l}$, and $\tilde{E}_1$ in \eqref{e1}, \eqref{e15} and \eqref{e100}, we have 
		$$
		\sum_{l=1}^d a_{i_l}^{t_1} b_{j_l}^{\tilde{E}_1-t_1}=\sum_{i\in \mathcal{J}_1} a_i^{t_1} \sum_{l\in I_i} b_{j_l}^{\tilde{E}_1-t_1} \leq\sum_{i\in \mathcal{J}_1} a_i^{t_1} = 1.
		$$
		Since 
		$\sum_{l=1}^d a_{i_l}^{t_1} b_{j_l}^{D_1-t_1}=1$, we have $D_1\leq \tilde{E}_1$. And $D_1=\tilde{E}_1$ if and only if 
		$$
		\sum_{l=1}^d a_{i_l}^{t_1} b_{j_l}^{\tilde{E}_1-t_1}=1,
		$$ which is equivalent to that $S_{1,l}=\tilde{E}_1-t_1$ for all $l$. This completes the proof of the lemma.
	\end{proof}

	\begin{proof}[Proof of  Theorem \ref{th3}-\eqref{e11}]
		Note that $\dimh K=\max\{D_1,D_2\}$, $\dima K=\max\{E_1,E_2\}$. It suffices to prove the case $\mathcal{J}_1 \times \mathcal{J}_2 \neq \mathcal{J}$ since the case $\mathcal{J}_1 \times \mathcal{J}_2 = \mathcal{J}$ is obvious.
		\vspace{0.2cm}
		
		``$\ref{A} \Longrightarrow \dimh K=\dima K$''. Without loss of generality, assume that $E_1\geq E_2$ and $K$ has uniform vertical fibres. So $\dima K =E_1 =\tilde{E}_1$, and $K$ is of horizontal or mixed type. Using Lemma \ref{le9}, we have 
		$E_1=D_1.$ Suppose that $D_2>D_1$, then by Lemma \ref{le3}, $\text{int}(\mathcal{S}_2)\neq \emptyset$, so $K$ is of mixed type. Again using Lemma \ref{le9}, we have 
		$$
		D_2\leq \tilde{E}_2 =E_2 \leq \dima K=E_1=D_1,
		$$
		a contradiction. Therefore, $D_1\geq D_2$, $\dimb K=\dima K$, and the condition \ref{B} holds. Combining this with Theorem \ref{th1}-\eqref{e42}, it follows that $\dimh K=\dimb K=\dima K$.
		\vspace{0.2cm}
		
		``$\dimh K =\dima K \Longrightarrow \dimb K =\dima K $''. This is obvious.
		\vspace{0.2cm}
		
		``$\dimb K =\dima K \Longrightarrow \ref{A}$''. Note that 
		$$
		\max \{D_1,D_2\}=\max \{E_1,E_2\}.
		$$
		When $E_1>E_2$, $E_1=\tilde{E}_1$ and $K$ is of horizontal or mixed type. By Lemma \ref{le9}, we find that $D_1\leq \tilde{E}_1=E_1.$ Suppose that $D_2>D_1$, by Lemma \ref{le3}, $K$ is of mixed type. By Lemma \ref{le9} again, we get 
		$$
		D_2 \leq \tilde{E}_2=E_2 < \dima K=\dimb K =D_2,
		$$
		a contradiction. So $D_2\leq D_1=E_1=\tilde{E}_1$. Still by Lemma \ref{le9}, $K$ has uniform vertical fibres.
		When $E_1<E_2$, a similar argument implies that $K$ has uniform horizontal fibres.
		When $E_1=E_2$, we have either $D_1=E_1=\tilde{E}_1$, or $D_2=E_2=\tilde{E}_2$. By Lemma \ref{le9}, $K$ has either uniform vertical or horizontal fibres. The above discussion implies that the condition \ref{A} is satisfied.
		
	\end{proof}
	
	Next, we prove that 
	$$
	``\ref{L}  \iff  \diml K=\dimh K \iff  \diml K=\dima K \iff K \text{ is Ahlfors regular} \text{''}.
	$$
	
	Recall that in \eqref{e108}, \eqref{e80} and by Lemma \ref{le5},
	\begin{equation}\nonumber \left\{
		\begin{aligned}
			F_1&=t_1+\min_{l\in \{1,\cdots, d\}} S_{1,l},\quad& & F_2=3, \quad & & \text{ if } K \text{ is of horizontal type},
			\\
			F_1&=3,\quad& & F_2=t_2+\min_{l\in \{1,\cdots, d\}} S_{2,l}, \quad & & \text{ if } K \text{ is of vertical type},
			\\
			F_1&=t_1+\min_{l\in \{1,\cdots, d\}} S_{1,l},\quad& & F_2=t_2+\min_{l\in \{1,\cdots, d\}} S_{2,l}, \quad & & \text{ if } K \text{ is of mixed type},
		\end{aligned} \right.
	\end{equation}
	and
	$$
	\begin{aligned}
		G_1&=\max_{\bm{q}\in \text{int}(\mathcal{S})}g_1(\bm{q})=\max_{\bm{q}\in \text{int}(\mathcal{S})}\frac{RR(\bm{q})}{RA(\bm{q})}+ \frac{QQ(\bm{q})-RR(\bm{q})}{SB(\bm{q})}, \quad
		\\
		G_2&=\max_{\bm{q}\in \text{int}(\mathcal{S})}g_2(\bm{q})=\max_{\bm{q}\in \text{int}(\mathcal{S})}\frac{SS(\bm{q})}{SB(\bm{q})}+\frac{QQ(\bm{q})-SS(\bm{q})}{RA(\bm{q})}.
	\end{aligned}
	$$
	Write 
	$$
	\tilde{F}_1=t_1+\min_{l\in \{1,\cdots, d\}} S_{1,l}, \quad \tilde{F}_2=t_2+\min_{l\in \{1,\cdots, d\}} S_{2,l}.
	$$
	\begin{lemma}\label{le10}
		For a \bara~carpet $K$, we have
		\begin{enumerate}[(a)]
			\item \label{e16}
			$\tilde{F}_1\leq G_1$, and if $\tilde{F}_1=G_1,$ then $K$ has uniform vertical fibres,
			\item \label{e17}
			$\tilde{F}_2\leq G_2$, and if $\tilde{F}_2=G_2$, then $K$ has uniform horizontal fibres.
		\end{enumerate}
	\end{lemma}
	\begin{proof}
		We only prove part \eqref{e16} here.
		\vspace{0.2cm}
		
		For $\bm{q}=(q_1,\cdots,q_d)\in \text{int}(\mathcal{S})$, note that 
		$$
		QQ(\bm{q})-RR(\bm{q})=\sum_{l=1}^d q_l \log \frac{q_l}{R_{i_l}(\bm{q})} =\sum_{i\in \mathcal{J}_1} R_i(\bm{q}) \sum_{l\in I_i}\frac{q_l}{R_i(\bm{q})} \log \frac{q_l}{R_i(\bm{q})},
		$$
		and 
		$$
		SB(\bm{q})=\sum_{l=1}^d q_l \log b_{j_l}=\sum_{i\in \mathcal{J}_1} R_i(\bm{{q}}) \sum_{l\in I_i} \frac{q_l}{R_i(\bm{q})}\log b_{j_l}.
		$$
		Since $\sum_{i\in \mathcal{J}_1} R_i(\bm{q})=1$, we have 
		\begin{equation}\label{e101}
			g_1(\bm{q})\geq \frac{RR(\bm{q})}{RA(\bm{q})}+\min_{i\in \mathcal{J}_1} \frac{\sum_{l\in I_i}\frac{q_l}{R_i(\bm{q})} \log \frac{q_l}{R_i(\bm{q})}}{\sum_{l\in I_i} \frac{q_l}{R_i(\bm{q})}\log b_{j_l}}:=h_1(\bm{q}).
		\end{equation}
		Noticing that for $i\in \mathcal{J}_1$, $\sum_{l\in I_i}\frac{q_l}{R_i(\bm{q})}=1$, we have 
		\begin{equation}\label{e105}
			\frac{RR(\bm{q})}{RA(\bm{q})}\leq t_1, \quad\text{ and }\quad \frac{ \sum_{l\in I_i}\frac{q_l}{R_i(\bm{q})} \log \frac{q_l}{R_i(\bm{q})}}{\sum_{l\in I_i} \frac{q_l}{R_i(\bm{q})}\log b_{j_l}}\leq S_{1,m}, \quad\text{ for each }m=1,\cdots, d \text{ with }i_m=i,
		\end{equation}
		which yields that 
		$$
		\sup_{\bm{q}\in \text{int}(\mathcal{S})} h_1(\bm{q})\leq \tilde{F}_1.
		$$
		Choosing $\bm{q}= \big(a_{i_l}^{t_1} b_{j_l}^{S_{1,l}} \big)_{l=1}^d\in \text{int}(\mathcal{S})$, we have  
		$$
		R_i(\bm{q})=a_i^{t_1}, \quad \text{ and }\quad \frac{q_l}{R_i(\bm{q})}=b_{j_l}^{S_{1,l}}, \quad \text{ for all }l\in I_i, i\in \mathcal{J}_1,
		$$
		and so $h_1(\bm{q})=\tilde{F_1}.$ Thus, 
		$$
		\max_{\bm{q}\in \text{int}(\mathcal{S})} h_1(\bm{q})= \tilde{F}_1,
		$$ 
		which gives that $\tilde{F}_1\leq G_1$.
		
		Next, assume that $\tilde{F}_1=G_1$. By \eqref{e101}, this implies that there exists $\bm{q}=(q_1,\cdots,q_d)\in \text{int}(\mathcal{S})$ with $h_1(\bm{q})=\tilde{F}_1$ such that
		\begin{equation}\label{e18}
			\frac{\sum_{l\in I_i}\frac{q_l}{R_i(\bm{q})} \log \frac{q_l}{R_i(\bm{q})}}{\sum_{l\in I_i} \frac{q_l}{R_i(\bm{q})}\log b_{j_l}} \text{ are the same for all }i\in \mathcal{J}_1,
		\end{equation}
		and $g_1(\bm{q})=G_1$. Then by \eqref{e105}, $\frac{RR(\bm{q})}{RA(\bm{q})}= t_1$ and each term in \eqref{e18} is equal to $\min\{S_{1,l}:l=1,2,\cdots,d\}=\tilde{F}_1-t_1$. Then by Lemma \ref{le5}, $\bm{q}$ has the following form:
		$$
		\bm{q}=\left(q_l\right)_{l=1}^d=\left( a_{i_l}^{\theta_1} b_{j_l}^{\lambda_1} \big( \sum_{m\in I_{i_l}} b_{j_m}^{\lambda_1} \big)^{\rho_1-1} \right)_{l=1}^d,
		$$
		where 
		$$
		\theta_1=\frac{RR(\bm{q})}{RA(\bm{q})}, \quad 
		\lambda_1=\frac{QQ(\bm{q})-RR(\bm{q})}{SB(\bm{q})}, \quad 
		\rho_1=\frac{RA(\bm{q})}{SB(\bm{q})}.		
		$$
		Since $\frac{RR(\bm{q})}{RA(\bm{q})}=t_1$ and $\theta_1+\lambda_1=g_1(\bm{q})=G_1=\tilde{F}_1$, we have $\theta_1=t_1$ and $\lambda_1=\tilde{F}_1-t_1$.
		Thus for $i\in \mathcal{J}_1$, and $l\in I_i$, we have 
		$$
		\frac{q_l}{R_i(\bm{q})}=\frac{b_{j_l}^{\tilde{F}_1-t_1}}{\sum_{m\in I_i}b_{j_{m}}^{\tilde{F}_1-t_1}}.
		$$
		Substituting this into \eqref{e18}, we get 
		$$
		\sum_{l\in I_i} b_{j_l}^{\tilde{F}_1-t_1}=1, \quad \text{ for all }i\in \mathcal{J}_1.
		$$
		Thus, $\tilde{F}_1=G_1$ implies that $K$ has uniform vertical fibres.
	\end{proof}
	
	\begin{lemma}\label{le20}
		For a \bara~capret $K$, if \ref{L} holds, then $K$ is Ahlfors regular.
	\end{lemma}
	\begin{proof}
		When $K$ is of horizontal type, since condition \ref{L} holds, $K$ has uniform vertical fibres, i.e. there is $t\geq 0$ such that 
		\begin{equation}\label{e81}
			\sum_{l\in I_i} b_{j_l}^t=1 \quad \text{ for all }i\in \mathcal{J}_1.
		\end{equation}
		Combining this with  \eqref{e1} and \eqref{e10}, we find that $t=D_1-t_1$. Let $\bm{q}=\big(a_{i_l}^{t_1} b_{j_l}^{D_1-t_1}\big)_{l=1}^d\in \mathcal{S}$. Let $\mu$ be the self-affine measure associated with $\bm{q}$. For $\varepsilon\in \Sigma^{\mathbb{N}}$ and $\delta>0$, write $k_1=k_1(\varepsilon,\delta)$ and $k_2=k_2(\varepsilon,\delta)$ for short. Noticing that
		$$
		\delta\leq B_{\varepsilon|_{k_2}} \leq A_{\varepsilon|_{k_2}}, \quad \text{ and }\quad k_1=\max\{k\in \mathbb{N}:\delta\leq A_{\varepsilon|_k}\},
		$$
		we conclude that $k_1\geq k_2$ and $Q(\varepsilon,\delta)$ is of $1$-type. Combining this with Lemma \ref{le12} and \eqref{e81}, we have 
		$$
		\mu \big( Q(\varepsilon,\delta) \big) = \prod_{l=1}^{k_2} a_{i_l}^{t_1} b_{j_l}^{D_1-t_1} \prod_{l=k_2+1}^{k_1} a_{i_l}^{t_1}=A_{\varepsilon|_{k_1}}^{t_1} B_{\varepsilon|_{k_2}}^{D_1-t_1}.
		$$
		By Lemma \ref{le11}, there is a constant $C>0$ such that 
		$$
		\begin{aligned}
			\mu \big( B(\Pi(\varepsilon),\delta) \big) &\leq C \max_{Q\in \mathcal{Q}_{\delta}} \mu(Q) \leq C^{D_1+1} \delta^{D_1},\quad \big(\text{recall }\eqref{e112}\big)
			\\
			\mu \big( B(\Pi(\varepsilon),\delta) \big) &\geq \mu \left(Q(\varepsilon,\frac{\delta}{\sqrt{2}C}) \right) \geq \left(\frac{1}{\sqrt{2}C}\right)^{D_1} \delta^{D_1}.
		\end{aligned}
		$$
		From \cite[Proposition 2.1.5]{BSS23}, we conclude that $K$ is Ahlfors regular.
		
		When $K$ is of vertical type, it is similar as above.
		
		When $K$ is of mixed type. Let $\bm{q}^{(1)}=\big(a_{i_l}^{t_1} b_{j_l}^{D_1-t_1}\big)_{l=1}^d$, $\bm{q}^{(2)}=\big(b_{j_l}^{t_2} a_{i_l}^{D_2-t_2}\big)_{l=1}^d$ in $\mathcal{S}$. Let $\mu_1$ and $\mu_2$ be the self-affine measures associated with $\bm{q}^{(1)}$ and $\bm{q}^{(2)}$, respectively. Since $K$ has both uniform vertical and horizontal fibres, we see that 
		$$
		\sum_{l\in I_i} b_{j_l}^{D_1-t_1}=1, \quad \sum_{l\in J_j} a_{i_l}^{D_2-t_2}=1, \quad \text{ for all } i\in \mathcal{J}_1, j\in \mathcal{J}_2.
		$$
		For $\varepsilon\in \Sigma^{\mathbb{N}}$ and $\delta>0$, write $k_1=k_1(\varepsilon,\delta)$ and $k_2=k_2(\varepsilon,\delta)$ for short. If $Q(\varepsilon,\delta)$ is of $1$-type, we have 
		$$
		\begin{aligned}
			\mu_1 \big( Q(\varepsilon,\delta) \big) &= \prod_{l=1}^{k_2} a_{i_l}^{t_1} b_{j_l}^{D_1-t_1} \prod_{l=k_2+1}^{k_1} a_{i_l}^{t_1}=A_{\varepsilon|_{k_1}}^{t_1} B_{\varepsilon|_{k_2}}^{D_1-t_1},
			\\
			\mu_2 \big( Q(\varepsilon,\delta) \big) &= \prod_{l=1}^{k_2} b_{j_l}^{t_2} a_{i_l}^{D_2-t_2} \prod_{l=k_2+1}^{k_1} \Big(a_{i_l}^{D_2-t_2} \sum_{m\in I_{i_l}} b_{j_m}^{t_2} \Big) \leq A_{\varepsilon|_{k_1}}^{D_2-t_2} B_{\varepsilon|_{k_2}}^{t_2};
		\end{aligned}
		$$
		and if $Q(\varepsilon,\delta)$ is of $2$-type, we have 
		$$
		\begin{aligned}
			\mu_1 \big( Q(\varepsilon,\delta) \big) &= \prod_{l=1}^{k_1} a_{i_l}^{t_1} b_{j_l}^{D_1-t_1} \prod_{l=k_1+1}^{k_2} \Big(b_{j_l}^{D_1-t_1} \sum_{m\in J_{j_l}} a_{i_m}^{t_1}\Big)\leq A_{\varepsilon|_{k_1}}^{t_1} B_{\varepsilon|_{k_2}}^{D_1-t_1},
			\\
			\mu_2 \big( Q(\varepsilon,\delta) \big) &= \prod_{l=1}^{k_1} b_{j_l}^{t_2} a_{i_l}^{D_2-t_2} \prod_{l=k_1+1}^{k_2} b_{j_l}^{t_2} = A_{\varepsilon|_{k_1}}^{D_2-t_2} B_{\varepsilon|_{k_2}}^{t_2}.
		\end{aligned}
		$$
		Since by \ref{L}, $E_1=F_1=F_2=E_2$, and so by Lemma \ref{le9},
		$D_1=D_2$. Let $\mu=(\mu_1+\mu_2)/2$. Using Lemma \ref{le11}, we have 
		$$
		\frac{1}{2} \delta^{D_1} \leq \mu \big(Q(\varepsilon,\delta) \big) \leq C^{D_1} \delta^{D_1},
		$$
		which implies that $K$ is Ahlfors regular.
		
		This finishes the proof.
	\end{proof}

	\begin{proof}[Proof of Theorem \ref{th3}-\eqref{e12}]
		Note that $\diml A=\min \{F_1,F_2\}$.
		\vspace{0.2cm}
		
		``$\ref{L} \Longrightarrow \diml K=\dima K$''. When $K$ is of horizontal type, $K$ has uniform vertical fibres. So 
		$$
		\tilde{F}_1=t_1+\min_{l\in \{1,\cdots, d\}} S_{1,l} =t_1+\max_{l\in \{1,\cdots, d\}} S_{1,l} =\tilde{E}_1.
		$$
		Noticing that $F_1=\tilde{F}_1$ and $E_1=\tilde{E}_1$, we have 
		$$
		\diml K=F_1=E_1=\dima K.
		$$
		When $K$ is of vertical type, we also have $\diml K=\dima K$. When $K$ is of mixed type, $K$ has both uniform vertical and horizontal fibres. So 
		$$
		F_1=\tilde{F}_1=\tilde{E}_1=E_1, \quad F_2=\tilde{F}_2=\tilde{E}_2=E_2.
		$$
		Together with $F_1=F_2$,
		this implies that $\diml K =\dima K$.
		\vspace{0.2cm}
		
		``$\diml K=\dima K \Longrightarrow \diml K=\dimh K$''. This is obvious.
		\vspace{0.2cm}
		
		``$\diml K =\dimh K \Longrightarrow \ref{L}$''. By Lemma \ref{le7}, we have 
		\begin{equation}\label{e19}
			\min \{F_1,F_2\}=\max \{ G_1,G_2\}.
		\end{equation}
		When $K$ is of horizontal type, $\mathcal{S}_1=\mathcal{S}$, $\min\{F_1,F_2\}=\diml K=F_1=\tilde{F}_1$, and $\max\{G_1,G_2\}=\dimh K=G_1.$ Thus, by \eqref{e19}, $\tilde{F}_1=G_1.$ Using Lemma \ref{le10}, we conclude that $K$ has uniform vertical fibres. When $K$ is of vertical type, similarly $K$ has uniform horizontal fibres. When $K$ is of mixed type, we have 
		$$
		F_1=\tilde{F}_1, \quad F_2=\tilde{F_2}.
		$$
		Combining this with \eqref{e19} and Lemma \ref{le10}, we see that $F_1=\tilde{F}_1=G_1$, $F_2=\tilde{F}_2=G_2$, and $F_1=F_2$. Still by Lemma \ref{le10}, we also have $K$ has both uniform vertical and horizontal fibres.  This implies \ref{L} holds.
		\vspace{0.2cm}
		
		``$\ref{L} \Longrightarrow K \text{ is Ahlfors regular}$''. This follows by Lemma \ref{le20}.
		\vspace{0.2cm}
		
		``$K \text{ is Ahlfors regular} \Longrightarrow \diml K=\dima K$''. This follows by \cite[Theorem 6.4.1]{F21}.
	\end{proof}

	\section{Examples}\label{sec8}

	In this section, we provide two examples. Through Example \ref{ex1}, we point out that  for a \bara~carpet $K$ with $\dimh K <\dimb K$, it may happen that $0<\mathcal{H}^{\dimh K}(K)<+\infty$, which is unlike the setting of Lalley and Gatzouras. However, Corollary \ref{cor1} states that $\dimh K=\dimb K$ is still a sufficient condition for $0<\mathcal{H}^{\dimh K}(K)<+\infty$. We will provide the proof of Corollary \ref{cor1} after Example \ref{ex1}.
	In Example \ref{ex2}, we demonstrate that for a \bara~carpet $K$, it may happen that  
	\begin{equation}\label{e113}
	\diml K<\dimh K =\dimb K <\dima K, 
	\end{equation} 
	which completes the analysis of eight possible comparisons among lower, Hausdorff, box, and Assouad dimensions in Corollary \ref{cor4}. It is worth mentioning that the relationship in \eqref{e113} is not the first time it appears. Indeed, in 2021, Bárány, Käenmäki and Yu \cite{BKY21} provided an example of a non-carpet self-affine set $K$ satisfying \eqref{e113}.
	
	\begin{example}\label{ex1}
		Let $a=0.0765,b=0.2298,c=0.499,d=0.2904$. Consider the following IFS $\{\psi_1,\cdots,\psi_6\}$ (see left of Figure \ref{f1}):
		$$
		\begin{aligned}
			\psi_1\begin{pmatrix}
				x\\y
			\end{pmatrix}
			&=\begin{pmatrix}
				a&0\\
				0&d
			\end{pmatrix}
			\begin{pmatrix}
				x\\y
			\end{pmatrix}+
			\begin{pmatrix}
				0.05\\0.025
			\end{pmatrix},\quad& & &
			\psi_2\begin{pmatrix}
				x\\y
			\end{pmatrix}
			&=\begin{pmatrix}
				a&0\\
				0&d
			\end{pmatrix}
			\begin{pmatrix}
				x\\y
			\end{pmatrix}+
			\begin{pmatrix}
				0.05\\0.35
			\end{pmatrix},
			\\
			\psi_3\begin{pmatrix}
				x\\y
			\end{pmatrix}
			&=\begin{pmatrix}
				b&0\\
				0&d
			\end{pmatrix}
			\begin{pmatrix}
				x\\y
			\end{pmatrix}+
			\begin{pmatrix}
				0.2\\0.025
			\end{pmatrix},& & &
			\psi_4\begin{pmatrix}
				x\\y
			\end{pmatrix}
			&=\begin{pmatrix}
				b&0\\
				0&d
			\end{pmatrix}
			\begin{pmatrix}
				x\\y
			\end{pmatrix}+
			\begin{pmatrix}
				0.2\\0.675
			\end{pmatrix},
			\\
			\psi_5\begin{pmatrix}
				x\\y
			\end{pmatrix}
			&=\begin{pmatrix}
				c&0\\
				0&d
			\end{pmatrix}
			\begin{pmatrix}
				x\\y
			\end{pmatrix}+
			\begin{pmatrix}
				0.48\\0.35
			\end{pmatrix},& & &
			\psi_6\begin{pmatrix}
				x\\y
			\end{pmatrix}
			&=\begin{pmatrix}
				c&0\\
				0&d
			\end{pmatrix}
			\begin{pmatrix}
				x\\y
			\end{pmatrix}+
			\begin{pmatrix}
				0.48\\0.675
			\end{pmatrix}.
			\\
		\end{aligned}
		$$
		Denote $K$ the corresponding \bara~carpet. It is straightforward to see that $K$ has uniform vertical fibres but does not have uniform horizontal fibres. Using software ``Wolfram Mathematica'', we find that the numerical maximums of the functions $g_1$, $g_2$ are approximately
		$$
		G_1 \approx 1.368858891,\quad G_2\approx 1.368381784,
		$$
		and by solving the numerical solutions of equation \eqref{e10}, 
		$$
		D_1 \approx 1.368858891,\quad D_2 \approx 1.369071220.
		$$
		From $G_1>G_2$, we see that \ref{H'} holds, then combining Theorem \ref{th1} and Proposition \ref{le15}, we deduce that $$
		0<\mathcal{H}^{\dimh K}(K)<+\infty.
		$$
		From $D_2>D_1$, we see that \ref{B} does not hold, then still by Theorem \ref{th1},
		$$
		\dimh K<\dimb K.
		$$
		Note that here $D_1=G_1$ indeed, which follows from Remark \ref{re3}.
	\end{example}
	
	\begin{figure}[htbp]
		\centering
		\subfigure{
			\includegraphics[width=0.4\textwidth]{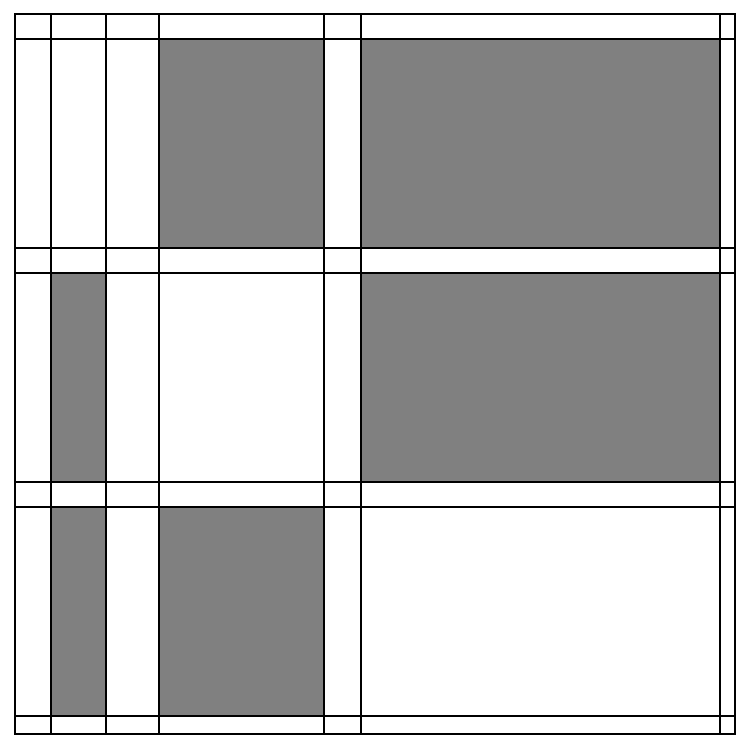}
			\begin{picture}(0,0)
				\put(-165,-5){$a$}
				\put(-125,-5){$b$}
				\put(-55,-5){$c$}
				\put(-185,25){$d$}
				\put(-185,85){$d$}
				\put(-185,140){$d$}
			\end{picture}
		}
		\subfigure{
			\includegraphics[width=0.4\textwidth]{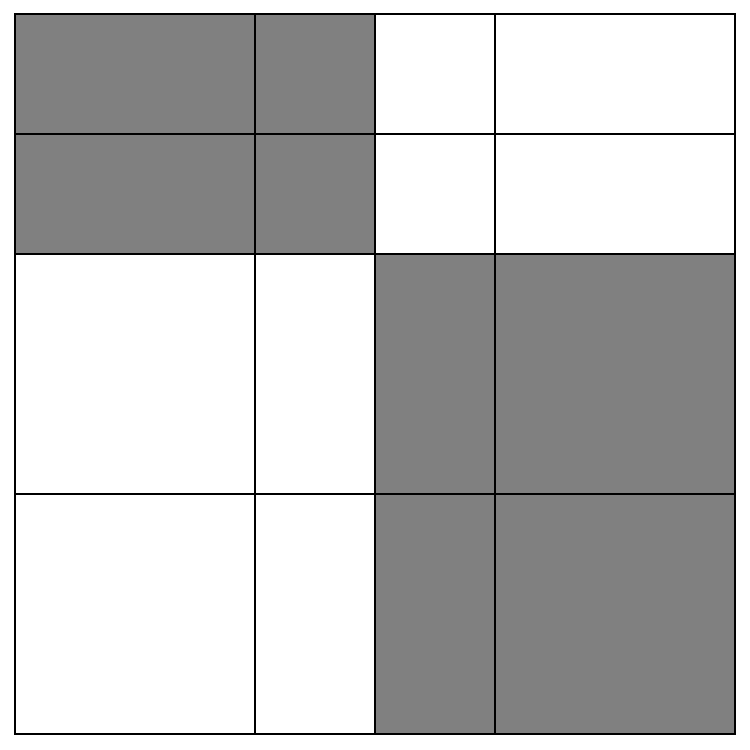}
			\begin{picture}(0,0)
				\put(-80,-7){$\frac{1}{6}$}
				\put(-40,-7){$\frac{1}{3}$}
				\put(-6,32){$\frac{1}{3}$}
				\put(-6,88){$\frac{1}{3}$}
				\put(-184,128){$\frac{1}{6}$}
				\put(-184,155){$\frac{1}{6}$}
				\put(-152,178){$\frac{1}{3}$}
				\put(-110,178){$\frac{1}{6}$}
			\end{picture}
		}
		
		\caption{Generating templatex for \bara~carpets in  Examples \ref{ex1} (\textit{left}) and  \ref{ex2} (\textit{right}).}
		\label{f1}
	\end{figure}
	\begin{proof}[Proof of Corollary \ref{cor1}]
		Due to Example \ref{ex1}, it suffices to prove that $\dimh K=\dimb K$ implies $0<\mathcal{H}^{\dimh K }(K)<+\infty$. Assume that $\dimb K=\dimh K$, then by Theorem \ref{th1}, the condition \ref{B} holds. Without loss of generality, assume that $D_1\geq D_2$, and $K$ has vertical uniform fibres. By Theorem \ref{th1} and Proposition \ref{le15}, it suffices to prove that \ref{H'} holds. We only need to look at the case $G_1\leq G_2$ and $\mathcal{J}_1\times \mathcal{J}_2 \neq \mathcal{J}$. By Lemma \ref{le7}, we have 
		\begin{equation}\label{e104}
			G_2=\dimh K=\dimb K=\max\{D_1,D_2\}\geq D_2.
		\end{equation}
		Using Proposition \ref{le8}, Lemmas \ref{le3} and \ref{le4}, we obtain that
		$$G_2= \max_{\bm{q} \in \text{int}(\mathcal{S}_2)} g_2(\bm{q}) \leq \max_{\bm{q}\in \mathcal{S}} f_2(\bm{q}) = D_2.
		$$ 
		Combining this with \eqref{e104}, we have $G_2=D_2$. Then by Lemma \ref{le19}, $K$ has uniform horizontal fibres. So \ref{H'} holds.
	\end{proof}

	\begin{example}\label{ex2}
		Let $K$ be a \bara~carpet whose IFS consisting of eight mappings  as shown in the right of Figure \ref{f1}. It is straightforward to see that $K$ has uniform horizontal fibres but does not have uniform vertical fibres, and $t_1=t_2=1$. Let 
		$\alpha$ be the unique solution of 
		$$
		\left(\frac{1}{6}\right)^{\alpha}+\left(\frac{1}{3}\right)^{\alpha}=1.
		$$
		It is directly to check that  $D_1=D_2=E_2=F_2=1+\alpha$, and 
		$$
		E_1=1+\frac{\log 2}{\log 3}>1+\alpha,\quad F_1 =1+\frac{\log 2}{\log 6}<1+\alpha.
		$$
		Therefore, \ref{B} holds and \ref{A}, \ref{L} do not hold. So by Theorems \ref{th1} and  \ref{th3},
		$$
		\diml K <\dimh K =\dimb K <\dima K .
		$$
	\end{example}

	
	\vspace{0.2cm}
	


	\bibliographystyle{amsplain}
	
\end{document}